\documentclass[a4paper,11pt]{amsart}
\usepackage{a4wide}
\advance\topmargin by -1.5\baselineskip
\advance\textheight by 3\baselineskip
\usepackage{mathtools}
\usepackage{enumerate}
\usepackage[utf8]{inputenc}
\usepackage{url}
\usepackage{pst-node}
\usepackage{booktabs}          
\usepackage{float}             
\usepackage{cite}
\usepackage{hyperref}
\usepackage{comment}
\usepackage{array}

\numberwithin{equation}{section}

\newcommand{\field}[1]{\mathbb{{#1}}}
\newcommand{\ideal}[1]{\mathfrak{{#1}}}
\newcommand{\Path}[1]{\mathcal{{#1}}}

\newcommand{\eps}{\varepsilon}
\newcommand{\R}{\field{R}}
\newcommand{\RG}{{\ideal{R}}}
\newcommand{\C}{\field{C}}
\renewcommand{\L}{\field{{L}}}
\newcommand{\MC}{\mathcal{M}_C}

\newcommand{\Q}{\field{{Q}}}
\newcommand{\K}{\field{{K}}}
\newcommand{\E}{\field{{E}}}
\newcommand{\F}{\field{{F}}}
\newcommand{\LC}{{\mathcal{L}}}

\newcommand{\DK}{{\disc_\K}}
\newcommand{\DL}{{\disc_\L}}

\newcommand{\lDK}{{\log\DK}}

\newcommand{\DLK}{{\disc_{\L/\K}}}
\newcommand{\DLF}{{\disc_{\L/\F}}}
\newcommand{\omDLK}{{\mathfrak{n}}}
\newcommand{\lDL}{{\log\DL}}
\newcommand{\llDL}{{\log\lDL}}
\newcommand{\nE}{{n_\E}}
\newcommand{\nK}{{n_\K}}
\newcommand{\nL}{{n_\L}}
\newcommand{\nLsq}{{n_\L^2}}
\newcommand{\ee}{{e}}
\newcommand{\OO}{\mathcal{O}}
\newcommand{\OK}{{\OO_\K}}
\newcommand{\Rp}{\Path{{R}}}
\newcommand{\Cp}{\Path{{C}}}
\newcommand{\cond}{{\ideal{f}}}
\newcommand{\p}{{\ideal{p}}}
\newcommand{\q}{{\ideal{q}}}
\renewcommand{\P}{{\ideal{P}}}
\newcommand{\DSG}{D}
\newcommand{\Sg}{{\mathbf{S}}}
\newcommand{\iI}{{\ideal{I}}}

\newcommand{\intpart}[1]{\left\lfloor#1\right\rfloor}

\newcommand{\disc}{\Delta}
\newcommand{\Artin}[1]{\genfrac{[}{]}{0pt}{1}{\L/\K}{#1}}

\newcommand{\dd}{\,\mathrm{d}}

\newcommand{\Norm}{\textrm{\upshape N}}

\DeclareMathOperator{\Imm}{Im}
\DeclareMathOperator{\Ree}{Re}

\DeclareMathOperator{\atan}{atan}
\DeclareMathOperator{\Ind}{Ind}

\DeclareMathOperator{\Gal}{Gal}

\newcommand{\mltc}{\multicolumn}

\newcounter{Lcount}
\renewcommand{\theLcount}{\roman{Lcount}}
\newenvironment{List}{\begin{list}{\theLcount.}{
\usecounter{Lcount}\setlength{\labelwidth}{1cm} \setlength{\leftmargin}{2em} }
}{\end{list}}     

\newtheorem{theorem}{Theorem}[section]
\newtheorem{lemma}[theorem]{Lemma}
\newtheorem*{lemma*}{Lemma}

\newtheorem{corollary}[theorem]{Corollary}
\newtheorem*{corollary*}{Corollary}

\theoremstyle{remark}

\newtheorem*{remark*}{Remark}

\newtheorem*{acknowledgements}{Acknowledgements}

\newcounter{stepcounterbl}
\numberwithin{stepcounterbl}{subsection}
\renewcommand{\thestepcounterbl}{\arabic{stepcounterbl}}
\def\step#1{\refstepcounter{stepcounterbl}\paragraph*{\bf Step \thestepcounterbl: #1}}

\newcounter{fixedtab}
\newenvironment{fixedtab}{
\refstepcounter{fixedtab}
\def\caption##1{\textsc{Table \thefixedtab}: ##1\\[2ex]}%
\noindent\begin{minipage}{\textwidth}\begin{center}%
}
{%
\end{center}\end{minipage}%
}

\allowdisplaybreaks[4]

\begin{document}
\title[An explicit Chebotarev density theorem]{An explicit Chebotarev density theorem under GRH}

\author[L.~Greni\'{e}]{Lo\"{\i}c Greni\'{e}}
\address[L.~Greni\'{e}]{Dipartimento di Ingegneria Gestionale, dell'Informazione e della Produzione\\
         Universit\`{a} di Bergamo\\
         viale Marconi 5\\
         24044 Dalmine
         Italy}
\email{loic.grenie@gmail.com}

\author[G.~Molteni]{Giuseppe Molteni}
\address[G.~Molteni]{Dipartimento di Matematica\\
         Universit\`{a} di Milano\\
         via Saldini 50\\
         I-20133 Milano\\
         Italy}
\email{giuseppe.molteni1@unimi.it}



\begin{abstract}
We prove an explicit version of the Chebotarev theorem for the density of prime ideals with fixed Artin
symbol, under the assumption of the validity of the Riemann hypothesis for the Dedekind zeta functions.
In appendix we also give some explicit formulas counting non-trivial zeros of Hecke's $L$-functions,
in that case without assuming the truth of the Riemann hypothesis.
\end{abstract}

\maketitle

\begin{flushright}
{\sl To Jurek Kaczorowski\\ for his 60th birthday}
\end{flushright}

\section{Introduction}\label{sec:1}
In order to state the results we need to fix some notation. Thus, given a number field $\K$ we denote
$\nK$ its dimension and $r_1(\K)$, $r_2(\K)$ the number of its real, respectively imaginary places; the
absolute value of its discriminant is denoted as $\DK$, $\p$ always denotes a nonzero prime ideal of the
integer ring $\OK$, and $\Norm\p$ its absolute norm; $\Lambda_\K$ denotes the analogue of the von
Mangoldt function, i.e. the function which is defined on the set of ideals of $\OK$ and whose value
at an ideal $\iI$ is $\log\Norm\p$ if $\iI=\p^m$ for some $\p$ and $m\geq 1$, and zero otherwise.\\
Moreover, let $\K\subseteq \L$ be a Galois extension of number fields with relative discriminant $\DLK$,
and let $\P$ be a prime ideal of $\L$ above a non-ramified $\p$ prime ideal of $\OK$. Then the Artin
symbol $\Artin{\p}$ denotes the conjugacy class of the Frobenius automorphism corresponding to $\P/\p$,
and which is extended multiplicatively on the prime powers in $\OK$ coprime to $\DLK$.\\
Let $C$ be any conjugacy class in $G:=\Gal(\L/\K)$ and let $\eps_C$ be its characteristic function. Then
the function $\pi_C$ and the Chebyshev function $\psi_C$ are defined as
\begin{align*}
\pi_C(x)  &:= \sharp \big\{\p\colon \p\text{ non-ramified in }\L/\K,\Norm\p \leq x,\Artin{\p}=C\big\}                         \\
          &\phantom{:}= \sum_{\substack{\p\\\p\text{ non-ram.}\\\Norm\p\leq x}}\eps_C\big(\Artin{\p}\big),                    \\
\psi_C(x) &:= \sum_{\substack{\iI\subset\OK\\\iI\text{ non-ram.}\\\Norm\iI\leq x}}\eps_C\big(\Artin{\iI}\big)\Lambda_\K(\iI).
\end{align*}
The first function counts the number of non-ramified prime ideals with prescribed Artin symbol, while
Chebyshev's function does the same but with a suitable logarithmic weight supported on prime powers.
The celebrated Chebotarev density theorem states that $\pi_C(x)\sim \frac{|C|}{|G|}\frac{x}{\log x}$ when
$x$ diverges, a claim which can be stated equivalently by saying that $\psi_C(x) \sim \frac{|C|}{|G|}x$.\\
We introduce also two other functions which are closely related to $\pi_C$ and $\psi_C$ but that are
easier to deal with. They are built using an arithmetical function which comes from the theory of Artin
$L$-functions and extends $\eps_C\big(\Artin{\p}\big)$ to ramifying prime ideals. To wit, for any prime
ideal $\p\subseteq \OK$ (possibly ramified) let $\P$ be any prime ideal dividing $\p\mathcal{O}_\L$, let
$I$ be the inertia group of $\P$ and $\tau$ be one of the $|I|$ Frobenius automorphisms corresponding to
$\P/\p$. Let
\begin{equation}\label{eq:1.1}
\theta(C;\p^m)  := \frac{1}{|I|}\sum_{a\in I}\varepsilon_C(\tau^m a).
\end{equation}
Notice that $\theta(C;\p^m)\in[0,1]$, and that for non-ramified primes it is $1$ if and only if $\tau^m$
belongs to $C$, and $0$ otherwise. We define
\begin{align*}
\pi(C;x)       &:= \sum_{\substack{\p\colon\Norm\p\leq x}}\theta(C;\p) \log\Norm\p,             \\
\psi(C;x)      &:= \sum_{\substack{\iI\subset\OK\\\Norm\iI\leq x}}\theta(C;\iI)\Lambda_\K(\iI).
\end{align*}
Observe that $\psi_C(x)$ and $\psi(C;x)$ agree except on ramified-prime-powers ideals, being
\begin{equation}\label{eq:1.2}
\psi(C;x) = \psi_C(x) + \RG_C(x)
\end{equation}
with
\begin{equation}\label{eq:1.3}
\RG_C(x)
:= \sum_{\p\mid\DLK}\sum_{\substack{m\geq 1\\\Norm\p^m\leq x}}\theta(C;\p^m)\log\Norm\p.
\end{equation}
In particular, $0\leq \psi_C(x)\leq \psi(C;x)$ for every $x$, so that every upper bound for $\psi(C;x)$
gives also a bound for $\psi_C(x)$, and a lower bound for $\psi_C(x)$ produces a lower bound for
$\psi(C;x)$.
\smallskip\\
Jeffrey Lagarias and Andrew Odlyzko~\cite{LagariasOdlyzko} provided versions of Chebotarev's theorem
which are explicit in their dependence on the field $\K$ up to positive universal constants which however
are not estimated, and Joseph Oesterl\'{e}~\cite{Oesterle} announced that
\begin{equation}\label{eq:1.4}
\Big|\frac{|G|}{|C|}\psi(C;x)-x\Big|\leq \sqrt{x}\Big[\Big(\frac{\log   x}{ \pi}+2\Big)\lDL
                                                     +\Big(\frac{\log^2 x}{2\pi}+2\Big)\nL
                                                 \Big]
\qquad
\forall\,x\geq 1
\end{equation}
under the assumption of the generalized Riemann hypothesis.
On the other hand, Lowell Schoenfeld~\cite{Schoenfeld1} proved that the Riemann hypothesis implies that
\[
|\psi_\Q(x) - x| \leq \frac{1}{8\pi}\sqrt{x}\log^2 x
\qquad
\forall\,x\geq 59.
\]
%
%
(He states this result for $x\geq 73.2$, but actually it is easy to check that the inequality holds also
for $x\in [59,73.2]$). This result shows that it should be possible to improve the constants appearing in
Oesterl\'{e}'s result. Bruno Winckler~\cite[Th.~8.1]{Winckler} proved a result similar
to~\eqref{eq:1.4}, but with larger coefficients of logs in the $\log\disc_\L$ and $n_\L$ parts.\\
In~\cite{GrenieMolteni3} we have proved an analogue of Schoenfeld's result for the easier case $\K=\L$,
where all prime ideals are counted.
In this paper we generalize this work to the full set of extensions and classes, as in Oesterl\'{e}'s result,
but with the improved constants. In fact, the following theorem is our main result.

\begin{theorem}\label{th:1.1}
Assume GRH holds. Then $\forall x\geq 1$
\begin{align}
\Big|\frac{|G|}{|C|}\psi(C;x)-x\Big|
\leq \sqrt{x}\Big[\Big(\frac{\log   x}{2\pi}+2\Big)\lDL
                + \Big(\frac{\log^2 x}{8\pi}+2\Big)\nL
             \Big],                                         \notag         \\
\Big|\frac{|G|}{|C|}\psi_C(x)-x\Big|
\leq \sqrt{x}\Big[\Big(\frac{\log   x}{2\pi}+2\Big)\lDL
                + \Big(\frac{\log^2 x}{8\pi}+2\Big)\nL
             \Big].                                         \label{eq:1.5}
\end{align}
\end{theorem}
%
From the proof it will be clear that the constants $+2$ have nothing special and other values are
possible. For instance, one can prove that
\[
\Big|\frac{|G|}{|C|}\psi(C;x)-x\Big|
\leq \sqrt{x}\Big[\Big(\frac{\log   x}{2\pi}+2\Big)\lDL
                + \frac{\log^2 x}{8\pi}\nL
             \Big]
     + 40,
\]
again for all $x\geq 1$.
%
Moreover, the $+40$ can be removed if $\nL\geq 7$, and both $+2$, $+40$ can be removed if $x$ is large
enough. One can also prove a result of the form of~\cite[Corollary~1.3]{GrenieMolteni3} where $\log x$ is
substituted by $\log\big(\frac{c x}{\log^2x}\big)$ for some constant $c$. All remarks apply also to
$\psi_C(x)$.
\smallskip\\
By partial summation one deduces the following result.
\begin{corollary}\label{cor:1.2}
Assume GRH holds. Then $\forall x\geq 2$
\begin{align*}
\Big|\frac{|G|}{|C|}\pi(C;x)-\int_2^x \frac{\!\dd u}{\log u}\Big|
\leq& \sqrt{x}\Big[\Big(\frac{     1}{2\pi} + \frac{3}{\log x}\Big)\lDL
                 + \Big(\frac{\log x}{8\pi} + \frac{1}{4\pi} + \frac{6}{\log x}\Big)\nL
              \Big],                                                                      \\
\Big|\frac{|G|}{|C|}\pi_C(x)-\int_2^x \frac{\!\dd u}{\log u}\Big|
\leq& \sqrt{x}\Big[\Big(\frac{     1}{2\pi} + \frac{3}{\log x}\Big)\lDL
                 + \Big(\frac{\log x}{8\pi} + \frac{1}{4\pi} + \frac{6}{\log x}\Big)\nL
              \Big].
\end{align*}
\end{corollary}
\noindent
This corollary also could be improved in the secondary terms as in~\cite[Corollary~1.4]{GrenieMolteni3}
which, unfortunately, was stated incorrectly and should read
\begin{corollary*}[{\!\!\cite[Corollary~1.4]{GrenieMolteni3}}]
Assume GRH holds. Then $\forall x \geq 2$
\begin{multline*}
\Big|\pi_\K(x)- \int_2^x \frac{\dd u}{\log u}\Big|\\
\leq
\sqrt{x}\Big[ \Big(\frac{1}{2\pi}-\frac{\log\log x}{ \pi\log x} + \frac{5.8}{\log x}\Big)\log \disc_\K
             +\Big(\frac{1}{8\pi}-\frac{\log\log x}{2\pi\log x} + \frac{3.6}{\log x}\Big)n_\K\log x
             + 0.3
             + \frac{14  }{\log x}
        \Big].
\end{multline*}
\end{corollary*}
\smallskip

The general strategy for the proof is quite similar to the one of~\cite{LagariasOdlyzko}
and~\cite{GrenieMolteni2}. However, many estimations have to be done with special care, in order to
reduce the range of fields $\K$, extensions $\L/\K$ and $x$ where the claims have to be proved directly
via explicit computations.\medskip\\
We have made available at the address:\\
\url{http://users.mat.unimi.it/users/molteni/research/chebotarev/chebotarev.gp}\\
the PARI/GP~\cite{PARI2} code we have used to compute the constants in this paper.

\begin{acknowledgements}
We wish to thank Karim Belabas for comments and interesting discussions, and the referee for useful
comments and improvements in the text. The authors are members of the INdAM group GNSAGA.
\end{acknowledgements}

\section{Facts}\label{sec:2}
Let
\[
\psi^{(1)}(C;x) := \int_{0}^{x}\psi(C;t)\dd t.
\]
As observed by Ingham~\cite[Ch.~2, Sec. 5]{Ingham2}, since $\psi(C;x)$ is non-decreasing as a function of
$x$, one has the double inequality
\begin{equation}\label{eq:2.1}
\begin{array}{ll}
\displaystyle
\psi(C;x) \leq \frac{\psi^{(1)}(C;x+h)-\psi^{(1)}(C;x)}{h} \quad{} &\text{if }h>0,   \\
\displaystyle
\psi(C;x) \geq \frac{\psi^{(1)}(C;x+h)-\psi^{(1)}(C;x)}{h} \quad{} &\text{if }-x<h<0.
\end{array}
\end{equation}
We let, for $s>1$,
\begin{equation}\label{eq:2.2}
K(C;s) := \sum_{\iI\subseteq \OK}\theta(C;\iI)\Lambda_\K(\iI)(\Norm\iI)^{-s}.
\end{equation}
As in~\cite[Ch.~IV Sec.~4, p. 73]{Ingham2} and~\cite[Sec.~5]{LagariasOdlyzko}, we have the integral
representation
\[
\psi^{(1)}(C;x) = \frac{1}{2\pi i}\int_{2-i\infty}^{2+i\infty} K(C;s)\frac{x^{s+1}}{s(s+1)}\,\dd s.
\]
Let $g$ be any element in $C$, then the orthogonality of the irreducible characters $\phi$ of $G$ allows
one to write
\[
\theta(C;\p^m) = \frac{|C|}{|G|}\sum_{\phi}\bar{\phi}(g)\phi_\K(\p^m)
\]
where
\[
\phi_\K(\p^m) := \frac{1}{|I|}\sum_{a\in I}\phi(\tau^m a).
\]
The definitions of $\theta(C;\cdot)$ and $\phi_\K$ are modelled on the definition of the Artin
$L$-functions $L(s,\phi,\L/\K)$, giving the equality
\[
K(C;s)
= \sum_{\iI\subseteq \OK}\theta(C;\iI)\Lambda_\K(\iI)(\Norm\iI)^{-s}
= -\frac{|C|}{|G|}\sum_\phi \bar\phi(g) \frac{L'}{L}(s,\phi,\L/\K)
\]
for $\Ree s > 1$.\\
Following an argument of Lagarias and Odlyzko (which comes from Deuring~\cite{Deuring1} and
Mac\-Cluer~\cite{MacCluer}) we can modify the identity in order to use only Hecke $L$-functions,
for which the continuation as holomorphic functions (apart at $s=1$)
in $\C$ is proved:
it is~\cite[Lemma~4.1]{LagariasOdlyzko}, but a quick review can
be useful.\\
As above, let $g$ be any fixed element in $C$. Let $H$ be the cyclic group generated by $g$ and let $\E:=
\L^H=\L\strut^g$, the subfield of $\L$ fixed by $H$. Let $f_g\colon H\to \C$ be the characteristic
function of $\{g\}$. A direct computation shows that it induces on $G$ the class function
$\Ind_H^Gf_g\colon G\to \C$ whose values are
\[
(\Ind_H^Gf_g)(y)
= \frac{1}{|H|}\sum_{s\in G}f_g(s^{-1}y s)
= \begin{cases}
   \frac{|G|}{|C||H|} & \text{if $y\in C$}\\
   0                  & \text{otherwise}.
 \end{cases}
\]
Thus, the characteristic function of $C$ is $\frac{|C||H|}{|G|}\Ind_H^Gf_g$. By orthogonality of
characters $\chi$ of $H$ one has
\[
f_g = \frac{1}{|H|}\sum_{\chi}\bar{\chi}(g)\chi,
\]
thus
\[
(\Ind_H^Gf_g)(y) = \frac{1}{|H|}\sum_{\chi}\bar{\chi}(g)(\Ind_H^G\chi)(y),
\]
and the characteristic function of $C$ is now written as
$\frac{|C|}{|G|}\sum_{\chi}\bar{\chi}(g)\Ind_H^G\chi$. Using the definition of $\theta(C;\cdot)$, we find
that
\[
\theta(C;\p^m) = \frac{|C|}{|G|}\sum_{\chi}\bar{\chi}(g)\chi_\K(\p^m)
\]
where
\[
\chi_\K(\p^m) := \frac{1}{|I|}\sum_{a\in I}(\Ind_H^G\chi)(\tau^m a).
\]
In this way we get
\begin{equation}
K(C;s)
 = -\frac{|C|}{|G|}\sum_\chi \bar\chi(g) \frac{L'}{L}(s,\Ind_H^G\chi,\L/\K)
 = -\frac{|C|}{|G|}\sum_\chi \bar\chi(g) \frac{L'}{L}(s,\chi,\L/\E),        \label{eq:2.3}
\end{equation}
which means
\begin{equation}\label{eq:2.4}
\psi^{(1)}(C;x)
= -\frac{|C|}{|G|}\sum_\chi \bar\chi(g)\frac{1}{2\pi i}\int_{2-i\infty}^{2+i\infty} \frac{L'}{L}(s,\chi,\L/\E)\frac{x^{s+1}}{s(s+1)}\,\dd s,
\end{equation}
where only abelian (i.e., Hecke, by class field theory) $L$-functions appear.

Thus, let $\E\subseteq \L$ be an abelian extension of fields and let $\chi$ be any irreducible character
of $\Gal(\L/\E)$. We will use $L(s,\chi)$ to denote $L(s,\chi,\L/\E)$. Also, set $\delta_\chi = 1$ if
$\chi$ is the trivial character, $0$ otherwise.

We recall that for each $\chi$ there exist uniquely determined non-negative integers
$a_\chi$, $b_\chi$ such that
\[
a_\chi + b_\chi = \nE
\]
and a positive integer $Q(\chi)$ such that if we define
\begin{equation}\label{eq:2.5}
\Gamma_\chi(s)
:= \Big[\pi^{-\frac{s}{2}}\Gamma\Big(\frac{s}{2}\Big)\Big]^{a_\chi}
   \Big[\pi^{-\frac{s+1}{2}} \Gamma\Big(\frac{s+1}{2}\Big)\Big]^{b_\chi}
\end{equation}
and
\begin{equation}\label{eq:2.6}
\xi(s,\chi) := [s(s-1)]^{\delta_\chi} Q(\chi)^{s/2}\Gamma_\chi(s)L(s,\chi),
\end{equation}
then $\xi(s,\chi)$ satisfies the functional equation
\begin{equation}\label{eq:2.7}
\xi(1-s,\bar\chi) = W(\chi)\xi(s,\chi),
\end{equation}
where $W(\chi)$ is a certain constant of absolute value $1$. For the trivial character $\chi$, the Hecke
$L$-function $L(s,\chi,\L/\E)$ coincides with Dedekind's zeta function $\zeta_\E(s)$, and in this case
$a_\chi=r_1(\E)+r_2(\E)$ and $b_\chi=r_2(\E)$. Furthermore, $\xi(s,\chi)$ is an entire function (by class
field theory) of order $1$ and does not vanish at $s = 0$, and hence by Hadamard's product theorem we have
\begin{equation}\label{eq:2.8}
\xi(s,\chi) = e^{A_\chi+B_\chi s} \prod_{\rho\in Z_\chi} \Big(1 - \frac s\rho\Big) e^{s/\rho}
\end{equation}
for some constants $A_\chi$ and $B_\chi$, where $Z_\chi$ is the set of zeros (multiplicity included) of
$\xi(s,\chi)$. They are precisely those zeros $\rho = \beta + i\gamma$ of $L(s,\chi)$ for which $0 <
\beta < 1$, the so-called ``non-trivial zeros'' of $L(s,\chi)$. From now on $\rho$ will denote a
non-trivial zero of $L(s,\chi)$.

Lastly, we introduce a special notation for the type of sum on characters as the one appearing
in~\eqref{eq:2.4}, and for any $f\colon \widehat{\Gal(\L/\E)}\to \C$ we set
\[
\MC f := \sum_\chi \bar\chi(g)f(\chi)
\]
where we recall that $g$ is a fixed element of $C$.

\section{Preliminary inequalities}\label{sec:3}
\subsection{Reduction to Dedekind Zeta functions}
Differentiating~\eqref{eq:2.6} and~\eqref{eq:2.8} logarithmically we obtain the identity
\begin{equation}\label{eq:3.1}
\frac{L'}{L}(s,\chi)
  = B_\chi
   + \sum_{\rho\in Z_\chi} \Big(\frac{1}{s-\rho}+\frac{1}{\rho}\Big)
   - \frac{1}{2}\log Q(\chi)
   - \delta_\chi\Big(\frac{1}{s}+\frac{1}{s-1}\Big)
   - \frac{\Gamma'_\chi}{\Gamma_\chi}(s),
\end{equation}
for all complex $s$.
Using~\eqref{eq:2.5}, \eqref{eq:2.6} and~\eqref{eq:3.1} one sees that
\begin{equation}\label{eq:3.2}
\begin{array}{ll}
\displaystyle\frac{L'}{L}(s,\chi) = \frac{a_\chi-\delta_\chi}{s} + r_\chi  + O(s)    & \text{as $s\to 0 $},\\[.4cm]
\displaystyle\frac{L'}{L}(s,\chi) = \frac{b_\chi}{s+1}           + r'_\chi + O(s+1)  & \text{as $s\to -1$},
\end{array}
\end{equation}
where
\begin{equation}\label{eq:3.3}
\begin{array}{ll}
\displaystyle
r_\chi  = B_\chi
         + \delta_\chi
         - \frac{1}{2} \log\frac{Q(\chi)}{\pi^{\nE}}
         - \frac{a_\chi}{2}\frac{\Gamma'}\Gamma(1)
         - \frac{b_\chi}{2}\frac{\Gamma'}\Gamma\Big(\frac{1}{2}\Big),\\[.4cm]
\displaystyle
r'_\chi =- \frac{L'}{L}(2,\bar{\chi})
         - \log\frac{Q(\chi)}{\pi^{\nE}}
         - \frac{\nE}{2}\frac{\Gamma'}\Gamma\Big(\frac{3}{2}\Big)
         - \frac{\nE}{2}\frac{\Gamma'}\Gamma(1).
\end{array}
\end{equation}
Comparing the previous formula for $r_\chi$ and~\eqref{eq:3.1}, we get
\begin{align*}
r_\chi =& \frac{L'}{L}(s,\chi)
        - \sum_{\rho\in Z_\chi}\frac{s}{\rho(s-\rho)}
        + \delta_\chi\Big(\frac{1}{s}+\frac{1}{s-1}\Big)    \\
       &+ \frac{a_\chi}{2}\Big(\frac{\Gamma'}{\Gamma}\Big(\frac{s}{2}\Big) - \frac{\Gamma'}{\Gamma}(1)\Big)
        + \frac{b_\chi}{2}\Big(\frac{\Gamma'}{\Gamma}\Big(\frac{s+1}{2}\Big) - \frac{\Gamma'}{\Gamma}\Big(\frac{1}{2}\Big)\Big)
\end{align*}
for every $s\in \C$. Setting $s=2$ this formula simplifies to
\begin{equation}\label{eq:3.4}
r_\chi = \frac{L'}{L}(2,\chi)
        - \sum_{\rho}\frac{2}{\rho(2-\rho)}
        + \frac{5}{2}\delta_\chi
        + b_\chi.
\end{equation}
We come back to the situation where $g\in C$ and $\E=\L\strut^g$, so that $\L/\E$ is a cyclic extension
for which $g$ is a generator of $\Gal(\L/\E)$. The following lemma computes the mean values of the
parameters $a_\chi$ and $b_\chi$ appearing in~\eqref{eq:2.5}. To simplify the formulas, we will write
from now on $r_1$ and $r_2$ for $r_1(\L)$ and $r_2(\L)$.
\begin{lemma}\label{lem:3.1}
Let
\[
\Sg:=
\begin{cases}
r_1 + r_2      & \text{if $g$ has order 1}, \\
r_2 - 2r_2(\E) & \text{if $g$ has order 2}, \\
0              & \text{otherwise},
\end{cases}
\]
and let $\delta_C$ defined to be $1$ if $C$ is the trivial class and $0$ otherwise. Then
\begin{align*}
\MC a_\chi &= \sum_{\chi}\bar{\chi}(g) a_\chi = \Sg,             \\
\MC b_\chi &= \sum_{\chi}\bar{\chi}(g) b_\chi = \delta_C\nE-\Sg = \delta_C\nL-\Sg.
\end{align*}
\end{lemma}
\begin{proof}
If $C$ is the trivial class, i.e. $g$ has order $1$, we have $\MC a_\chi=\sum_\chi a_\chi=r_1+r_2=\Sg$
because the extension $\L/\E$ is Galois, hence $\prod_\chi L(s,\chi)=\zeta_\L(s)$. We have as well
$\MC b_\chi=\sum b_\chi = r_2 = \nL-\Sg$, hence the result is proved. We henceforth assume that $g$ has
order at least $2$.\\
By duality, the set of characters of $\Gal(\L/\E)$ is cyclic: let $\varphi$ be a generator.
The character $\chi$ corresponds to a Hecke character $\tilde\chi$ of the id\`{e}les of $\E$. For any
real embedding of $\E$, let $p_\ell(\chi)$ be $1$ if the local component of $\tilde\chi$ at $\ell$ is the
sign character, and $0$ otherwise. We furthermore denote $s_\chi$ the number of $\ell$'s for which
$p_\ell(\chi)=1$. The construction of Hecke characters and $L$-functions shows that
$a_\chi = r_1(\E)+r_2(\E)-s_\chi$, see~\cite{Heilbronn2}. In particular,
\begin{align*}
\sum_{\chi}\bar{\chi}(g) a_\chi
= -\sum_{\chi}\bar{\chi}(g) s_\chi.
\end{align*}
For every fixed real embedding $\ell$ one has $p_\ell(\chi\chi')=p_\ell(\chi)+p_\ell(\chi')\pmod{2}$, thus
$s_\chi=0$ when $\chi$ is an even power of $\varphi$, and $s_\chi=s_\varphi$ otherwise. This shows that if
$|\Gal(\L/\E)|$ is odd, then $s_\chi=0$ for every character, while when $|\Gal(\L/\E)|$ is even one gets
\begin{align*}
\sum_{\chi}\bar{\chi}(g) a_\chi
= -s_\varphi\bar{\varphi}(g)\sum_{k=0}^{|\Gal(\L/\E)|/2-1}(\bar{\varphi}^2)^{k}(g).
\end{align*}
This is the sum on the subgroup of the square characters, thus it is zero unless $\varphi^2(g)=1$. This
happens if and only $|\Gal(\L/\E)|=2$, because $g$ is a generator, and in this case $\varphi(g)=-1$. Thus
we get:
\begin{align*}
\sum_{\chi}\bar{\chi}(g) a_\chi
= \begin{cases}
    s_\varphi & \text{if $|\Gal(\L/\E)|=2$}\\
    0         & \text{otherwise}.
  \end{cases}
\end{align*}
To conclude, we have $p_\ell(\varphi)=1$ if and only if $\ell$ ramifies in $\L/\E$ hence
$s_\varphi=r_2(\L)-2r_2(\E)$. This proves the lemma for the sum of the $a_\chi$'s. For the sum of the
$b_\chi$'s it is sufficient to observe that
\[
\sum_\chi\bar\chi(g)(a_\chi+b_\chi) = \sum_\chi\bar\chi(g)\nE = 0.
\qedhere
\]
\end{proof}
Note that if $g$ has order $1$, then $\Sg=r_1+r_2=\frac{\nL+r_1}{2}$. In the other cases we have $0\leq
\Sg\leq r_2 = \frac{\nL-r_1}{2}$. Thus in all cases $0\leq \Sg\leq \frac{\nL-r_1}{2}+\delta_Cr_1$.

\begin{lemma}\label{lem:3.2}
Let $\L/\E$ be a cyclic extension and let $Z$ be the multiset of non-trivial zeros of the Dedekind zeta
function $\zeta_\L$. Let $f$ be any complex function with $\sum_{\rho\in Z} |f(\rho)|<\infty$. Then
\[
\MC\sum_{\rho\in Z_{\chi}} f(\rho)
 = \sum_{\rho\in Z} \epsilon(\rho) f(\rho)
\]
where, for any $\rho\in Z$, $|\epsilon(\rho)|=1$ and $\epsilon(\overline{\rho}) =
\overline{\epsilon(\rho)}$.
\end{lemma}
\begin{proof}
Since $\zeta_\L = \prod_\chi L(s,\chi)$, the multiset $Z$ is the disjoint union of the multisets
$Z_\chi$. Moreover, for each $\rho$ in $Z$ there is a well defined character $\chi$ such that $\rho\in
Z_\chi$; for this $\rho$ we set $\epsilon(\rho):= \overline{\chi(g)}$. This rule respects the formula
$\epsilon(\overline{\rho}) = \overline{\epsilon(\rho)}$, because $\rho$ belongs to $Z_\chi$ if and only
if $\bar\rho$ belongs to $Z_{\bar{\chi}}$. Thus, we can write
\[
\MC\sum_{\rho\in Z_{\chi}} f(\rho)
= \sum_\chi\sum_{\rho\in Z_{\chi}} \bar\chi(g)f(\rho)
= \sum_{\rho\in Z} \epsilon(\rho) f(\rho).
\]
The equality $|\epsilon(\rho)|=1$ is obvious.
\end{proof}

\begin{lemma}\label{lem:3.3}
Let $Z$ be the multiset of non-trivial zeros of the Dedekind zeta function $\zeta_\L$. Recall that $\L/\E$
is a cyclic extension and that $\Sg$ and $\epsilon(\rho)$ are defined in Lemmas~\ref{lem:3.1}
and~\ref{lem:3.2}, respectively. We have
\[
\MC r_\chi = 2\sum_{\rho\in Z} \frac{\epsilon(\rho)}{\rho(2-\rho)}
             - \frac{\nL}{\nK|C|}\sum_{\iI\subseteq \OK}\theta(C;\iI)\Lambda_\K(\iI)(\Norm\iI)^{-2}
             + \nL\delta_C
             - \Sg
             + \frac{5}{2}.
\]
\end{lemma}
\begin{proof}
By~\eqref{eq:3.4} and Lemmas~\ref{lem:3.1} and~\ref{lem:3.2} we get
\begin{equation}\label{eq:3.5}
\MC r_\chi = 2\sum_{\rho\in Z} \frac{\epsilon(\rho)}{\rho(2-\rho)}
             + \MC\frac{L'}{L}(2,\chi)
             + \nL\delta_C
             - \Sg
             + \frac{5}{2}.
\end{equation}
Moreover, by~\eqref{eq:2.3} we have
\[
\MC\frac{L'}{L}(2,\chi) = - \frac{|G|}{|C|} K(C;2)
\]
hence by~\eqref{eq:2.2} we have
\begin{equation}\label{eq:3.6}
\MC\frac{L'}{L}(2,\chi)
= -\frac{\nL}{\nK|C|} \sum_{\iI\subseteq \OK}\theta(C;\iI)\Lambda_\K(\iI)(\Norm\iI)^{-2}.
\end{equation}
The result follows from~\eqref{eq:3.5} and~\eqref{eq:3.6}.
\end{proof}

\begin{lemma}\label{lem:3.4}
We have
\[
|\MC r'_\chi|
\leq \lDL
   + \nL \Big|\frac{\zeta'}{\zeta}\Big|(2)
   + (\log 2\pi + \gamma-1)\nL \delta_C,
\]
where $\gamma=0.5772\dots$ is the Euler--Mascheroni constant.
\end{lemma}
\begin{proof}
By~\eqref{eq:3.3} we get
\[
\MC r'_\chi
 = - \MC \frac{L'}{L}(2,\bar{\chi})
   - \MC \log Q(\chi)
   + \nE\Big(\log \pi - \frac{1}{2}\frac{\Gamma'}\Gamma\Big(\frac{3}{2}\Big) - \frac{1}{2}\frac{\Gamma'}\Gamma(1)\Big)\MC 1.
\]
Replacing $C$ by $C_1=[g^{-1}]$ and $g$ by $g^{-1}$ in~\eqref{eq:2.3} and conjugating, we get
\[
\MC\frac{L'}{L}(2,\bar{\chi}) = - \frac{|G|}{|C|} \overline{K(C^{-1};2)}
\]
which by~\eqref{eq:2.2} is estimated by $\frac{\nL}{|C|} \big|\frac{\zeta'}{\zeta}\big|(2)$ because
$0\leq \theta(C;\cdot)\leq 1$ by definition. Moreover,
\[
|\sum_\chi \bar{\chi}(g)\log Q(\chi)|
\leq \sum_\chi \log Q(\chi)
=    \lDL,
\]
by the product formula for conductors.\\
The result follows because $\frac{\Gamma'}\Gamma\big(\frac{3}{2}\big) + \frac{\Gamma'}\Gamma(1) = 2 -
\log 4 - 2\gamma$ and $\nE\MC 1 = \nL\delta_C$.
\end{proof}

\begin{lemma}\label{lem:3.5}
We define, for any $x>1$ and any character $\chi$,
\begin{align*}
f_1(x)    &:= \sum_{r=1}^{\infty}\frac{x^{1-2r}}{2r(2r-1)},
\qquad
f_2(x)     := \sum_{r=2}^{\infty}\frac{x^{2-2r}}{(2r-1)(2r-2)},     \\
R_\chi(x) &:= - (a_\chi-\delta_\chi)(x\log x-x)
              + b_\chi    (\log x + 1)
              - a_\chi     f_1(x)
              - b_\chi     f_2(x)
\intertext{and}
R_C(x)    &:= \MC R_\chi(x).
\end{align*}
Then for any $x>1$,
\begin{align*}
R_C(x) &=\int_0^x\log u\dd u - \Sg\int_1^{x+1}\log u\dd u
         + \delta_C\frac{\nL}{2}\Big[\log(x^2-1) + x \log\Big(\frac{x+1}{x-1}\Big)\Big],\\
R'_C(x)&=\log x - \Sg\log(x+1) + \delta_C\frac{\nL}{2}\log\Big(\frac{x+1}{x-1}\Big).
\end{align*}
\end{lemma}
\begin{proof}
We have
\begin{align*}
f_1(x) &= \frac{1}{2}\Big[x\log(1-x^{-2}) + \log\Big(\frac{1+x^{-1}}{1-x^{-1}}\Big)\Big],   \\
f_2(x) &= 1-\frac{1}{2}\Big[\log(1-x^{-2}) + x\log\Big(\frac{1+x^{-1}}{1-x^{-1}}\Big)\Big].
\end{align*}
%
%
Assume first that $C$ is not the trivial class. By Lemma~\ref{lem:3.1},
\begin{align*}
R_C(x) &= -(\Sg - 1)(x\log x-x)
        - \Sg(\log x+1)
        - \Sg f_1(x)
        + \Sg f_2(x)                                                        \\
       &= x\log x-x
        + \Sg\Big({-}(x+1)\log x
                  + x
                  - \frac{x+1}{2}\Big(\log(1-x^{-2})
                                    + \log\Big(\frac{1+x^{-1}}{1-x^{-1}}\Big)
                                 \Big)
             \Big)                                                          \\
       &= x\log x-x
        + \Sg(-(x+1)\log x
              + x
              - (x+1)\log(1+x^{-1})
             )                                                              \\
       &= x\log x-x
        + \Sg(x-(x+1)\log(x+1)),
\end{align*}
which produces the formulas for $R_C$ and $R'_C$ stated in the lemma for a non-trivial class. For the
trivial class we have to add $\nL$ times
\[
1 + \log x - f_2(x) = \frac{1}{2}\Big[\log(x^2-1) + x\log\Big(\frac{x+1}{x-1}\Big)\Big]
\]
to $R_C$ and $\frac{1}{2}\log\big(\frac{x+1}{x-1}\big)$ to its derivative.
\end{proof}

\subsection{Bounds for the ramification term}
\begin{lemma}\label{lem:3.6}
Let $x\geq 1$.
Then
\[
\RG_C(x) \leq \min\Big(\frac{|C|}{p},1\Big)\omDLK\log x
\]
where $p$ is the smallest prime dividing $|G|$, and $\omDLK := \sum_{\p\mid\DLK}1$ is the number of prime
ideals of $\K$ dividing $\DLK$.
\end{lemma}
\begin{proof}
From its definition~\eqref{eq:1.3} we have
\begin{align*}
\RG_C(x)
&\leq \max_{\substack{\p\mid\Delta_{\L/\K}\\m\geq 1}}(\theta(C;\p^m))\sum_{\p\mid\DLK}\log\Norm\p\sum_{\substack{m\geq 1\\\Norm\p^m\leq x}}1
 =    \max_{\substack{\p\mid\Delta_{\L/\K}\\m\geq 1}}(\theta(C;\p^m))\sum_{\p\mid\DLK}\log\Norm\p \intpart{\frac{\log x}{\log\Norm\p}}      \\
&\leq \max_{\substack{\p\mid\Delta_{\L/\K}\\m\geq 1}}(\theta(C;\p^m))\omDLK\log x,
\end{align*}
and~\eqref{eq:1.1} immediately shows that $\theta(C;\p^m)\leq \min(|C|/|I|,1)$. The proof concludes
because the order of the inertia group is at least $p$ for ramified primes.
\end{proof}

\begin{lemma}\label{lem:3.7}
Let $\omDLK = \sum_{\p\mid\DLK}1$ as in Lemma~\ref{lem:3.6}. We have the following bounds:
\begin{List}
\item \label{eq:lem3.7_1}  If $\L\neq\Q[\sqrt{\pm 3}]$ and $\L\neq\Q[\sqrt{\pm 15}]$ then $\omDLK \leq
\frac{\lDL}{\log 4}$.
\item \label{eq:lem3.7_2} If $\nL=3$, the bound improves to $\omDLK \leq \frac{\lDL}{\log 49}$.
\item \label{eq:lem3.7_3}If $\nL/\nK$ is not prime, the bound improves to $\omDLK \leq \frac{\lDL}{\log 22}$
except for the quartic fields of discriminant in $\{144,225,400,441,3600,7056,176400\}$ (twenty five
fields in total).
\item \label{eq:lem3.7_4} If $\lDL>e^{1.1714}\,\nK$, then
\[
\omDLK\leq \frac{\lDL}{\llDL - \log\nK - 1.1714}.
\]
\end{List}
\end{lemma}
The proof will make clear that Item~\ref{eq:lem3.7_4} is valid even when $\L/\K$ is not Galois. Moreover,
the inequality $\lDL>e^{1.1714}\,\nK$ holds except for just a few fields when $\L\neq\K$. Precisely, the
only exceptions for $\nK=1$ are the fields $\L$ with $\DL\leq 25$ (i.e., the cubic field of discriminant
$-23$ and seventeen quadratic fields),
%
for $\nK=2$ they are the twenty four quartic fields with $\DL\leq 634$,
%
for $\nK=3$ the four sextic fields with $\DL\leq 15986$.
There are no exceptions with $\nK\geq 4$.
%
%
\begin{proof}
We can assume $|G|\geq 2$ otherwise $\omDLK=0$.\\
{\sc Item} {\it \ref{eq:lem3.7_1}.}\\
Suppose $\K\neq \Q$. We split the set of primes dividing $\disc_{\L/\K}$ into three (possibly empty)
sets: $\{\p_i\}_{i=1}^a$, $\{\ideal{q}_j\}_{j=1}^b$ and $\{\ideal{s}_\ell\}_{\ell=1}^c$, which are the
set of primes whose norm is $2$, $3$ and $\geq 4$, respectively. Note that $a$, $b\leq \nK$. Then
\[
\disc_\L
= \disc_\K^{[\L:\K]} \Norm(\disc_{\L/\K})
= \disc_\K^{\nL/\nK} \Norm(\prod_i\p_i\prod_j\ideal{q}_j\prod_\ell\ideal{s}_\ell)
\geq \disc_\K^2 2^a 3^b 4^c.
\]
Moreover by Minkowski's bound we know that $\disc_\K^{1/\nK}\geq \sqrt{3}$, i.e. $\disc_\K^2\geq 3^\nK$.
Thus we get
\[
\DL
\geq 3^{\nK} 2^a 3^b 4^c
= 2^{\nK}\big(\tfrac{3}{2}\big)^{\nK} 2^a 3^b 4^c
\geq 2^a \big(\tfrac{3}{2}\big)^b 2^a 3^b 4^c
= 4^a \big(\tfrac{9}{2}\big)^b 4^c
\geq 4^{a+b+c}
= 4^{\omDLK}
\]
as claimed.\\
Suppose $\K=\Q$. Then $\omDLK = \omega(\DL)$. Let $p_j$, $j=2$, $3$, $\ldots$ be the sequence of primes.
Note that if $\DL\in[\prod_{k\leq j}p_k,\prod_{k\leq j+1}p_k)$ then
\[
\frac{\omDLK}{\lDL}
=    \frac{\omega(\DL)}{\lDL}
\leq \frac{j}{\log(\prod_{k\leq j}p_k)}
= \frac{j}{\vartheta(p_j)}.
\]
The sequence $\vartheta(p_j)/j$ is strictly increasing because it is the sequence of mean values of the
increasing sequence $\log p_j$.
%
Since $\frac{j}{\vartheta(p_j)}\leq 1/\log 4$ for $j=4$, and since $\prod_{k\leq 4}p_k = 210$, the
previous remark shows that $\omDLK\leq \lDL/\log 4$ as soon as $\DL\geq 210$. Moreover,
$\omega(\DL)\leq 3$ when $\DL\in[30,210)$. Thus in this range $\omDLK/\lDL\leq  3/\lDL$ so that it is
$\leq 1/\log 4$ as soon as $\DL\geq 4^3 = 64$. There are only 21 + 19 (resp. 4 + 1) quadratic (resp.
cubic) fields with $\DL< 64$; for all of them the inequality $\omDLK\leq \lDL/\log 4$ holds but for
$\Q[\sqrt{\pm 3}]$ and for $\Q[\sqrt{\pm 15}]$.

\noindent
{\sc Item} {\it \ref{eq:lem3.7_2}.}\\
Since $\L$ has to be a non-trivial Galois extension of $\K$, we must have $\K=\Q$ and $G$ cyclic of order
$3$. We thus know that the discriminant of $\L$ (hence $\DL$) is the square of an integer.
By~\cite{Hasse2} or~\cite[Th.~6.4.11, p.~341]{Cohen:Course}, the only primes that can divide $\DL$ are $3$
and the primes congruent to $1$ modulo $3$ and, if $3\mid\DL$ then $81\mid\DL$. This proves that $\DL\geq
49^{\omDLK}$, as needed.
\medskip\\
{\sc Item} {\it \ref{eq:lem3.7_3}.}\\
We prove that $\p^2\mid\disc_{\L/\K}$ for each prime ideal $\p\subseteq \OO_\K$ ramifying in $\L$. In
fact, we are assuming that $|G|$ is not a prime, thus $G$ has a proper subgroup and by Galois duality
there is a proper intermediate field $\F$, so that $\Q\subseteq \K\subset \F \subset \L$. Thus
\[
\disc_{\L/\K} = \disc_{\F/\K}^{[\L:\F]} \Norm_{\F/\K}(\DLF)
\]
Let $\p\subseteq \OO_\K$ be a prime ideal ramifying in $\L$. If $\p$ ramifies in $\F$, then
$\p^{[\L:\F]}\mid\disc_{\L/\K}$, hence $\p^2\mid\disc_{\L/\K}$.\\
Suppose now that $\p$ does not ramify in $\F$. Let $\P\subseteq \OO_\L$ be a prime above $\p$. As $\L/\K$
is Galois, it follows that $\q := \P\cap \F$ ramifies in $\L/\F$. Thus
$\q\mid\disc_{\L/\F}$. This proves that $\prod_{\q\mid\p\OO_\F}\q \mid\disc_{\L/\F}$. Hence
$\p\OO_\F\mid\disc_{\L/\F}$, because $\p\OO_\F=\prod_{\q\mid\p\OO_\F}\q$ (because $\p$ does not ramify in
$\F$, by hypothesis). Therefore $\p^{[\F:\K]} = \Norm_{\F/\K}(\p\OO_\F)\mid\disc_{\L/\K}$. In particular
$\p^2 \mid \disc_{\L/\K}$ also in this case.

Suppose $\K=\Q$. The previous computation shows that there exist integers $A$ and $B$ such that
$\DL = A^2B$ with $B$ squarefree and $B\mid A$. As a consequence
\[
\frac{\omDLK}{\log\DK}
= \frac{\omega(A^2B)}{\log(A^2B)}
\leq \frac{\omega(A)}{2\log A}
\]
and if $A\in[\prod_{k\leq j}p_k,\prod_{k\leq j+1}p_k)$ then
\[
\frac{\omDLK}{\log\DK}
\leq \frac{j}{2\vartheta(p_j)}.
\]
Since $\frac{j}{2\vartheta(p_j)}\leq 1/\log 22$ for $j=5$, and since $\prod_{k\leq 5}p_k = 2310$, the
previous remark shows that $\omDLK\leq \lDL/\log 22$ as soon as $A\geq 2310$. Moreover, $\omega(A)\leq 4$
when $A<2310$. Thus in this case $\omDLK/\lDL \leq 4/\lDL$ which is $\leq 1/\log 22$ as soon as $\DL\geq
22^4 = 234256$. Odlyzko's~Table~3 shows that $\DL\leq 234256$ is possible only for degrees $\nL\leq 7$,
and, given our hypothesis, it remains to test only $\nL=4$ and $\nL=6$. All quartic and sextics fields with
absolute discriminant up to $234256$ appear in megrez table: exploring the table we found that there are
only twenty five quartic fields which are Galois extensions of $\Q$ and which do not satisfy the bound
(they are the fields with discriminant in $\{144,225,400,441,3600,7056,176400\}$), and no sextic fields.\\
%
Suppose $\K\neq \Q$. We will prove that $\omDLK\leq \frac{\lDL}{\log 24}$. For $n=2$, $3$, $4$ let
$S_n$ be the set of prime ideals dividing $\disc_{\L/\K}$ and whose norm is $n$ and let $S_5$ be the set
of prime ideals dividing $\disc_{\L/\K}$ and whose norm is $\geq 5$. For all $2\leq n\leq 5$, let $a_n$ be
the cardinality of $S_n$. Then
\[
\DL
= \DK^{[\L:\K]} \Norm(\DLK)
\geq \DK^{\nL/\nK} (\Norm(\prod_{n=2}^{5}\prod_{\p \in S_n}\p))^2
\geq \DK^{\nL/\nK} (2^{a_2}3^{a_3}4^{a_4}5^{a_5})^2.
\]
Hence
\[
\lDL \geq \frac{\nL}{\nK}\lDK + 2\sum_{n=2}^{5} a_n \log n.
\]
The number appearing on the right-hand side is larger than $(\log 24)\sum_{n} a_n$ as soon as
\begin{equation}\label{eq:3.7}
\frac{\nL}{\nK}\lDK \geq \sum_{n=2}^{5} a_n\log(24/n^2).
\end{equation}
Note that $a_3\leq \nK$ and that $a_2+2a_4\leq \nK$ (because these primes factorize $2\OK$). As
$\nL/\nK\geq 4$, Inequality~\eqref{eq:3.7} holds for sure when
\[
\log(\DK^{1/\nK}) \geq \frac{1}{4}\log\Big(\frac{24^2}{2^2\cdot3^2}\Big) = \log 2,
\]
i.e. $\DK^{1/\nK}\geq 2$. The root discriminant of $\K$ satisfies this inequality for $\nK\geq 3$, as one
can see from line $b=1$ in Odlyzko's~Table~3.
%
For $\nK=2$ this is true for $\DK\geq 4$, thus $\K=\Q[\sqrt{-3}]$ is the unique exception to this
argument. However, in this case $S_2$ is empty and $a_3$, $a_4\leq 1$, thus the claim is true anyway.
\medskip\\
{\sc Item} {\it \ref{eq:lem3.7_4}.}\\
Set $p_0:=1$ and let $A\colon [0,+\infty)\to\R$ be the function such that
\[
\forall j\geq 0, \forall x\in [\vartheta(p_j),\vartheta(p_{j+1})),\quad
A(x):=\frac{x-\vartheta(p_j)}{\log p_{j+1}}+j,
\]
i.e., the continuous and piecewise affine map satisfying $A(\vartheta(p_j))=j$ for every $j$. It is an
increasing and concave map.\\
We also introduce on $(e^{1.1714},{+}\infty)$ the function $R(x):=\frac{x}{\log x-1.1714}$.
%
It is increasing for $x\geq x_R:=e^{2.1714}$, convex for $x\leq ex_R$ and concave for $x\geq ex_R$.\\
Guy Robin~\cite{Robin} proved that $\omega(n)\leq R(\log n)$ for all $n\geq 26$. As a consequence,
\[
\forall x>e^{1.1714},\quad A(x)\leq R(x).
\]
Indeed, $A(\vartheta(p_j))=j=\omega(\prod_{k=1}^j p_k)\leq R(\vartheta(p_j))$ when $j\geq 4$ by Robin's
result, and $A(ex_R) \leq R(ex_R)$, by explicit computation.
%
%
Thus, $A(x)\leq R(x)$ for $x\geq ex_R$ because $A$ is piecewise affine and $R$ is concave in this range.
On $(e^{1.1714},ex_R)$ the inequality still holds because $R$ is convex here and the tangent to its
graph in $ex_R$ stays above the graph of $A$.\\
%
%
Let $j_0:=\intpart{\omDLK/\nK}$ and $x_0:=\vartheta(p_{j_0})+(\omDLK/\nK-j_0)\log p_{j_0+1}$, so
that $\omDLK=A(x_0)\nK$.\\
Let $\p_j$, $j=1$, \ldots, $\omDLK$ be the primes ramifying in $\L/\K$. For each $j$ let $p_{k_j}$ be the
prime integer below $\p_j$ and $f_j$ be such that $\Norm(\p_j)=p_{k_j}^{f_j}$. We suppose that the ideals
are ordered such that the sequence $p_{k_j}$ is non-decreasing. We have
\[
\Norm(\DLK)
  =  \prod_{j=1}^{\omDLK} \Norm\p_j
  =  \prod_{j=1}^{\omDLK} p_{k_j}^{f_j}
\geq \prod_{j=1}^{\omDLK} p_{k_j}.
\]
For a given $p_k$, there are at most $\nK$ values of $j$ such that $p_k=p_{k_j}$, thus we get
\[
\DL\geq \Norm(\DLK) \geq \Big(\prod_{k=1}^{j_0} p_k\Big)^{\nK} p_{j_0+1}^{\omDLK - j_0\nK},
\]
so that $\lDL\geq x_0\nK$. Hence
\[
\frac{\omDLK}{\nK}
=    A(x_0)\leq A\Big(\frac{\lDL}{\nK}\Big)
\leq R\Big(\frac{\lDL}{\nK}\Big)
= \frac{1}{\nK}\frac{\lDL}{\log(\DL/\nK)-1.1714}
\]
when $\lDL> e^{1.1714}\nK$.
\end{proof}

\begin{lemma}\label{lem:3.8}
For every integer $n$, let $\tilde{\Lambda}_\L(n):=\sum_{\Norm\iI=n}\Lambda_\L(\iI)$. Then
for any $\ell\geq 1$ and any prime $p$ we have
\[
\sum_{r=1}^\nL \tilde{\Lambda}_\L(p^{\ell\nL +r})\geq \nL \log p.
\]
\end{lemma}
\begin{proof}
From the definition of $\tilde\Lambda_\L$, we have
\[
\sum_{r=1}^\nL \tilde{\Lambda}_\L(p^{\ell\nL +r})
= \sum_{r=1}^\nL \sum_{\p\mid p} \sum_{\substack{m\colon\\ \Norm\p^m = p^{\ell\nL +r}}} \log(\Norm\p)
= \sum_{\p\mid p} f_\p\Big(\sum_{r=1}^\nL \sum_{\substack{m\colon\\ mf_\p = \ell\nL +r}} 1\Big)\log p,
\]
where $f_\p$ is the inertia degree of $\p$ in the extension $\Q\subseteq\L$. To conclude, it is
sufficient to prove that
\[
\sum_{r=1}^\nL \sum_{\substack{m\colon\\ mf_\p = \ell\nL +r}} 1 \geq e_\p,
\]
where $e_\p$ is the ramification index of $\p$, because $\sum_{\p\mid p} f_\p e_\p = \nL$. To prove this
inequality, we pick $r\in\{1,...,f_\p\}$ such that $\ell\nL +r = 0 \pmod{f_\p}$. We then set
$m=(\ell\nL+r)/f_\p$, and this contributes by $1$ to the inner sum on $m$. We repeat this procedure in
the first $e_\p$ blocks of length $f_\p$: the claim follows since $e_\p f_\p\leq \nL$.
\end{proof}

\subsection{Bounds for sums on zeros of Dedekind Zeta functions}
\begin{lemma}\label{lem:3.9}
Assume GRH. Then we have
\begin{equation}\label{eq:3.8}
\sum_{|\gamma|\leq 2\pi} \frac{1}{|\rho|}
  + \sum_{|\gamma|>2\pi} \frac{|1/2+2\pi i|}{|\rho|^2}
\leq 1.348\lDL - 1.557\nL + 7.786 - 0.406r_1 - \ee_\nL,
\end{equation}
where the sums run over the non-trivial zeros $\rho=\frac{1}{2}+i\gamma$ of $\zeta_\L$. Here $\ee_\nL$ is
positive, with $\ee_1 \geq 5.529$, $\ee_2 \geq 0.751$ and $\ee_3 \geq 0.313$.
\end{lemma}
\begin{proof}
We prove this lemma with the same method of~\cite[Lemma 3.1]{GrenieMolteni3}. Thus, let
\[
g(\gamma) := \begin{cases}
                \frac{2}{(1+4\gamma^2)^{1/2}}   & \text{if }|\gamma|\leq 2\pi \\[1ex]
                \frac{2|1+4\pi i|}{1+4\gamma^2} & \text{otherwise}
             \end{cases}
\]
so that
\[
\sum_{|\gamma|\leq 2\pi}\frac{1}{|\rho|} + \sum_{|\gamma|>2\pi}\frac{|1/2+2\pi i|}{|\rho|^2}
  = \sum_{\gamma} g(\gamma).
\]
We observe that $g$ is continuous in $\R$. Moreover, let $f(s,\gamma):=4(2s-1)/((2s-1)^2+4\gamma^2)$ and
$f_\L(s) := \sum_{\gamma} f(s,\gamma)$. We look for a finite linear combination of $f(s,\gamma)$ at
suitable points $s_j$ such that
\begin{equation}\label{eq:3.9}
g(\gamma) \leq F(\gamma) := \sum_j a_j f(s_j,\gamma)
\qquad \forall\gamma\in\R,
\end{equation}
so that
\begin{equation}\label{eq:3.10}
\sum_{|\gamma|\leq 2\pi}\frac{1}{|\rho|} + \sum_{|\gamma|>2\pi}\frac{|1/2+2\pi i|}{|\rho|^2}
\leq \sum_j a_j f_\L(s_j).
\end{equation}
Once~\eqref{eq:3.10} is proved, we recover a bound for the sum on zeros recalling the identity
\begin{equation}\label{eq:3.11}
f_\L(s)
= 2\Ree\frac{\zeta'_\L}{\zeta_\L}(s)
    +\log\frac{\disc_\L}{\pi^{n_\L}}
    +\Ree\Big(\frac{2}{s}
             +\frac{2}{s-1}
        \Big)
    +(r_1+r_2)\Ree\frac{\Gamma'}{\Gamma}\Big(\frac{s}{2}\Big)
    +r_2\Ree\frac{\Gamma'}{\Gamma}\Big(\frac{s+1}{2}\Big).
\end{equation}
To determine a convenient set of constants $a_j$'s we set $s_j = 1+ j/2$ with $j=1,\ldots,23$,
\[
\Upsilon{:=}\{0.62,1,1.6,2.1,2.8,3.8,4.6,5.8,7.5,9.3,12.9,14,16,17,18,19,
20,30,40,50,10^2,10^3,10^4\},
\]
and we require:
\begin{enumerate}
\item $F(\gamma)=g(\gamma)$ for all $\gamma\in\Upsilon\cup\{0,2\pi\}$,
\item $F'(\gamma)=g'(\gamma)$ for all $\gamma\in\Upsilon$,
\item $\lim_{\gamma\to\infty}\gamma^2 F(\gamma)=\lim_{\gamma\to\infty}\gamma^2 g(\gamma)=|1/2+2\pi i|$.
\end{enumerate}
This produces a set of $49$ linear equations for the $49$ constants $a_j$'s ensuring~\eqref{eq:3.9}, at
least for $\gamma\in\Upsilon$. With an abuse of notation we take for $a_j$'s the
solution of the system, rounded above to $10^{-7}$: this produces the numbers in Table~\ref{tab:A2}. Then,
using Sturm's algorithm, we prove that the values found actually give an upper bound for $g$, so
that~\eqref{eq:3.9} holds with such $a_j$'s. These constants verify
\begin{equation}\label{eq:3.12}
\begin{aligned}
&\sum_{j=1}^{49} a_j                                                =     1.3479\ldots,\\
&\sum_{j=1}^{49} a_j\frac{\Gamma'}{\Gamma}\Big(\frac{s_j}{2}\Big)   \leq -0.421,
\end{aligned}
\qquad\qquad
\begin{aligned}
&\sum_{j=1}^{49} a_j\Big(\frac{2}{s_j}+\frac{2}{s_j-1}\Big)         \leq  7.786,\\
&\sum_{j=1}^{49} a_j\frac{\Gamma'}{\Gamma}\Big(\frac{s_j+1}{2}\Big) \leq  0.392.
\end{aligned}
\end{equation}
This suffices to manage all terms in~\eqref{eq:3.10} coming from all terms in~\eqref{eq:3.11} but the first
one. However, we observe that $a_1>0$, $a_2>0$ and the signs of the $a_j$'s alternate for $2\leq j\leq 49$.
We write $\sum_{j=1}^{49} a_j\frac{\zeta'_\L}{\zeta_\L}(s_j)$ as
\[
-\sum_n\tilde\Lambda_\L(n)S(n)\quad\text{with}\quad S(n):=\sum_{j=1}^{49} \frac{a_j}{n^{s_j}}.
\]
We isolate the first three terms in $S(n)$, and group the other ones by consecutive pairs
\[
S(n) = \Big(\frac{a_1}{n^{3/2}}
          + \frac{a_2}{n^{2}}
          + \frac{a_3}{n^{5/2}}\Big)
     + \Big(\frac{a_4}{n^{3}}
          + \frac{a_5}{n^{7/2}}\Big)
     + \Big(\frac{a_6}{n^{4}}
          + \frac{a_7}{n^{9/2}}\Big)
     + \cdots
     + \Big(\frac{a_{48}}{n^{25}}
          + \frac{a_{49}}{n^{51/2}}\Big).
\]
It is easy to verify that each group decreases for $n\geq 85597$, and that hence the same holds for
$S(n)$. A direct computation shows that $S(n+1)<S(n)$ holds also for $n\leq 85597$. Thus $S$ is a
decreasing sequence. Since $a_1>0$ we know that $S(n)>0$ definitively and hence always. Thus, we can
deduce that $-\ee_\nL := 2\sum_{j=1}^{49}a_j\frac{\zeta'_\L}{\zeta_\L}(s_j)=-2\sum_{n\geq 1}
\tilde\Lambda_\L(n) S(n)\leq 0$ which suffices to prove the claim for a generic $\nL$,
via~(\ref{eq:3.10}--\ref{eq:3.12}).\\
With the help of Lemma~\ref{lem:3.8} we can produce a better upper bound for $-\ee_\nL$, at least when
$\nL$ is small. In fact $S$ is decreasing, so that
\begin{align*}
\sum_{j=1}^{49}a_j\frac{\zeta'_\L}{\zeta_\L}(s_j)
&  =  -\sum_p\sum_m\tilde\Lambda_\L(p^m)S(p^m)
   =  -\sum_p\sum_{\ell=0}^\infty\sum_{r=1}^\nL\tilde\Lambda_\L(p^{\ell\nL+r})S(p^{\ell\nL+r})   \\
&\leq -\sum_p\sum_{\ell=0}^\infty\sum_{r=1}^\nL\tilde\Lambda_\L(p^{\ell\nL+r})S(p^{(\ell+1)\nL}).
\intertext{From Lemma~\ref{lem:3.8} and since $S\geq 0$, this is}
&\leq -\nL\sum_p\sum_{\ell=1}^\infty(\log p)S(p^{\ell\nL})
   =  -\nL\sum_p\sum_{\ell=1}^\infty\Lambda(p^\ell)S(p^{\ell\nL})
   =  \nL\sum_{j=1}^{49}a_j\frac{\zeta'}{\zeta}(s_j\nL).
\end{align*}
Hence
\[
-\ee_\nL
=    2\sum_{j=1}^{49}a_j\frac{\zeta'_\L}{\zeta_\L}(s_j)
\leq 2\nL\sum_{j=1}^{49}a_j\frac{\zeta'}{\zeta}(s_j\nL)
\]
whose value for $\nL=1$ is lower than $-5.529$, for $\nL=2$ is lower than $-0.751$ and for $\nL=3$ is
lower than $-0.313$ (the gain unfortunately decreases quickly: it is $-0.149$ for $\nL=4$ and only
$-0.074$ for $\nL=5$).
\end{proof}
\begin{lemma}\label{lem:3.10}
Assume GRH. Then one has
\[
\sum_{\rho} \frac{1}{|\rho(\rho+1)|} \leq 0.5375\lDL - 1.0355\nL + 5.3879 - 0.2635r_1,
\]
where the sum runs over the non-trivial zeros $\rho$ of $\zeta_\L$.
\end{lemma}
\begin{proof}
This claim is~\cite[Lemma~4.1]{GrenieMolteni2}, but now we repeat the computations keeping the extra term
which is proportional to $r_1$. Since
\[
\begin{aligned}
&\sum_j a_j                                                =     0.53747\ldots,\\
&\sum_j a_j\frac{\Gamma'}{\Gamma}\Big(\frac{s_j}{2}\Big)   \leq -0.6838,
\end{aligned}
\qquad\qquad
\begin{aligned}
&\sum_j a_j\Big(\frac{2}{s_j}+\frac{2}{s_j-1}\Big)         \leq  5.3879,\\
&\sum_j a_j\frac{\Gamma'}{\Gamma}\Big(\frac{s_j+1}{2}\Big) \leq -0.1567,
\end{aligned}
\]
the claim follows.
\end{proof}
We rewrite Theorem~\ref{th:A1} for $\E=\L$ and trivial character as
\begin{equation}\label{eq:3.13}
\Big|N_\L(T)-\frac{T}{\pi}\log\Big(\Big(\frac{T}{2\pi e}\Big)^{\nL}\DL\Big) - 2 + \frac{1}{4}r_1\Big|
\leq
c_1W_\L(T) + c_2\nL + c_3
\end{equation}
for every $T\geq T_0\geq 1$, where $W_\L(T) := \lDL+\nL\log(T/2\pi)$, $c_1=D_1$, $c_2=D'_2+D_1\log 2\pi$
and $c_3=D'_3$. With $T_0=2\pi$, the last line of Table~\ref{tab:A1} provides~\eqref{eq:3.13} with the
constants
\[
c_1 = 0.460,\quad c_2 = 2.491,\quad c_3 = 0.593.
\]
Other and smaller values for $c_1$ are available in Table~\ref{tab:A1}, but we need also a small value
for $c_2$ and $c_3$: this choice is adequate to our purpose. This proves
\begin{lemma}\label{lem:3.11}
For all $T\geq 2\pi$ one has
\begin{equation}\label{eq:3.14}
\sum_{|\gamma|\leq T} 1
= N_\L(T)
\leq T\Big(\frac{1}{\pi}+\frac{0.460}{T}\Big)W_\L(T) - T\Big(\frac{1}{\pi}-\frac{2.491}{T}\Big)\nL + 2.593-\frac{r_1}{4}.
\end{equation}
\end{lemma}
As in~\cite[Second sum]{GrenieMolteni3}, one has
\begin{lemma}\label{lem:3.12}
For all $T\geq 2\pi$ one has
\begin{equation}\label{eq:3.15}
\sum_{|\gamma|\geq T}\frac{1}{|\rho|^2}
\leq \Big(\frac{1}{\pi} + \frac{0.920}{T}\Big)\frac{W_\L(T)}{T}
    +\Big(\frac{1}{\pi} + \frac{5.220}{T}\Big)\frac{\nL}{T}
    +\frac{1.186}{T^2}.
\end{equation}
\end{lemma}
\begin{proof}
The proof remains the same in spite of the difference between the structure of~\eqref{eq:3.13} and
Trudgian's formula we used in~\cite{GrenieMolteni3} for this purpose, because the term $-1+r_1/4$
disappears in integrations. This provides the upper bound
\[
\sum_{|\gamma|\geq T}\frac{1}{|\rho|^2}
\leq \Big(\frac{1}{\pi} + \frac{2c_1}{T}\Big)\frac{W_\L(T)}{T}
   + \Big(\frac{1}{\pi} + \frac{\log 2\pi}{12T^2}\Big)\frac{\nL}{T}
   + \Big(2c_2+ \frac{c_1}{2}\Big)\frac{\nL}{T^2}
   + \frac{2c_3}{T^2},
\]
and the claim follows from the selected values of $c_j$'s.
\end{proof}
Note that the formula improves upon the one in~\cite{GrenieMolteni3} because now $c_1$, $c_2$ and $c_3$
are smaller.
\begin{lemma}\label{lem:3.13}
For all $T\geq 2\pi$ one has
\begin{align}
\smash[b]{\sum_{|\gamma|\leq T}}\frac{1}{|\rho|}
  +  \smash[b]{\sum_{|\gamma|>T}} \frac{|1+2\pi i|}{|\rho|^2}
\leq& \Big(\frac{1}{\pi}\log\Big(\frac{T}{2\pi}\Big) + 1.067 + \frac{2}{T}\Big)\lDL  \notag\\
&    + \Big(\frac{1}{2\pi}\log^2\Big(\frac{T}{2\pi}\Big) + \frac{2}{T}\log\Big(\frac{eT}{2\pi}\Big) - 1.633  - \frac{0.460}{T} + \frac{1.446}{T^2}\Big)\nL \notag\\
&    + 7.834
     - 0.406 r_1
     - \ee_\nL.                                                                      \label{eq:3.16}
\end{align}
\end{lemma}
\begin{proof}
Let~\eqref{eq:3.13} be written as $|N_\L(T) - A(T)| \leq R(T)$, with $A(T)$ representing the main term and
$R(T)$ the bound for the remainder term. To ease notations, we set $\ell:=|1/2+2\pi i|$.
We write
\begin{align*}
\sum_{|\gamma|\leq T}\frac{1}{|\rho|}
   + \sum_{|\gamma|>T} \frac{\ell}{|\rho|^2}
=& \sum_{|\gamma|\leq 2\pi} \frac{1}{|\rho|}
   + \sum_{|\gamma|>2\pi} \frac{\ell}{|\rho|^2}
   + \sum_{2\pi< |\gamma|\leq T}\Big(\frac{1}{|\rho|}-\frac{\ell}{|\rho|^2}\Big)            \\
\leq& \sum_{|\gamma|\leq 2\pi} \frac{1}{|\rho|}
   + \sum_{|\gamma|>2\pi} \frac{\ell}{|\rho|^2}
   + \sum_{2\pi< |\gamma|\leq T}\Big(\frac{1}{|\gamma|}-\frac{2\pi}{\gamma^2}\Big),
\end{align*}
where the last step follows by the general inequality
$\frac{1}{|1/2+i\gamma|}-\frac{\ell}{|1/2+i\gamma|^2} \leq \frac{1}{|\gamma|}-\frac{2\pi}{\gamma^2}$. By
partial summation we get
\begin{equation*}
\sum_{2\pi< |\gamma|\leq T}\!\!\!\Big(\frac{1}{|\gamma|}-\frac{2\pi}{\gamma^2}\Big)
\leq\!\!\!
     \int_{2\pi}^{T}\!\!\!\Big(\frac{1}{\gamma}-\frac{2\pi}{\gamma^2}\Big)\dd A(\gamma)
   + \frac{R(4\pi)}{4\pi}
   - \int_{2\pi}^{4\pi}\!\!\!\Big(\frac{1}{\gamma}-\frac{2\pi}{\gamma^2}\Big)R'(\gamma)\dd \gamma
   + \int_{4\pi}^{T}\!\!\!\Big(\frac{1}{\gamma}-\frac{2\pi}{\gamma^2}\Big)R'(\gamma)\dd \gamma
\end{equation*}
because $\frac{1}{\gamma}-\frac{2\pi}{\gamma^2}$ has a maximum at $4\pi$.
Since $R'(\gamma)= c_1\nL/\gamma$ this produces the bound
\[
\sum_{2\pi< |\gamma|\leq T}\!\!\!\Big(\frac{1}{|\gamma|}-\frac{2\pi}{\gamma^2}\Big)
\leq \int_{2\pi}^{T}\Big(\frac{1}{\gamma}-\frac{2\pi}{\gamma^2}\Big)\dd A(\gamma)
   + \frac{R(4\pi)}{4\pi}
   + c_1\Big(\frac{1}{8\pi}-\frac{1}{T}+\frac{\pi}{T^2}\Big)\nL.
\]
The claim follows from this bound, the equality
\[
\int_{2\pi}^{T}\Big(\frac{1}{\gamma}-\frac{2\pi}{\gamma^2}\Big)\dd A(\gamma)
 = \Big(\frac{1}{\pi}\log\Big(\frac{T}{2\pi}\Big)-\frac{1}{\pi}+\frac{2}{T}\Big)\lDL
    +\Big(\frac{1}{2\pi}\log^2\Big(\frac{T}{2\pi}\Big)+\frac{2}{T}\log\Big(\frac{eT}{2\pi}\Big)-\frac{1}{\pi}\Big)\nL,
\]
the result in~\eqref{eq:3.8} and the chosen values for the $c_j$'s constants.
\end{proof}

\section{A parametric result}\label{sec:4}
\begin{theorem}\label{th:4.1}
(GRH)
For every $x\geq 4$ and $T\geq 2\pi$ we have:
\begin{align}
\frac{|G|}{|C|}\psi(C;x) - x
&\leq L_a(x,T,\nL,\lDK),                                                 \label{eq:4.1}\\
-\Big(\frac{|G|}{|C|}\psi_C(x) - x\Big)
&\leq L_a(x,T,\nL,\lDK) + \DSG(x,T,\nL,\lDK) + \frac{|G|}{|C|}\RG_C(x),  \label{eq:4.2}
\end{align}
with
\begin{align*}
 L_a(x,T,n,\LC)
&:= F(x,T)\LC
     + G(x,T)n
     + H(x,T,n),                                                                                 \\[3mm]
 F(x,T) &:= \sqrt{x}
          \Big[\frac{1}{\pi}\log\Big(\frac{T}{2\pi}\Big)
               + 1.704
               + \frac{1.858}{T}
          \Big]
          + 1.075,                                                                               \\[3mm]
G(x,T) &:= \sqrt{x}
          \Big[\frac{1}{2\pi}\log^2\Big(\frac{T}{2\pi}\Big)
               + \Big(\frac{2}{\pi} + \frac{1.858}{T}\Big)\log\Big(\frac{T}{2\pi}\Big)
               - 1.633
               + \frac{7.729}{T}
          \Big]
          - 1.501,                                                                               \\[3mm]
H(x,T,n) &:= H_1(x,T)+H_2(x,T,n),                                                                \\[3mm]
H_1(x,T) &:= \frac{x+2}{T}
          + \sqrt{x}\Big(7.834 + \frac{3.779}{T}\Big)
          + 8.276,                                                                               \\[3mm]
H_2(x,T,n) &:= -\sqrt{x}\Big(\Big(0.406+\frac{1}{4T}\Big)r_1 + \ee_n\Big)
          + (1-\Sg)\log x
          + \Sg
          - 0.744n\delta_C
          - 0.527 r_1,                                                                           \\[3mm]
\DSG(x,T,n,\LC)
         &:= 2(\Sg-1)(\log x-1)
            + 3
            - 0.445 n
            + 2n\delta_C                                                                         \\
&\quad
          - \frac{\sqrt{x}}{T} \big(1.167 + 0.743\tfrac{\LC}{n} + 0.743\log\big(\tfrac{T}{2\pi}\big)\big)n.
\end{align*}
\end{theorem}
%
\begin{proof}
Following~\eqref{eq:2.4}, we consider for a character $\chi$ of $\Gal(\L/\E)$ the integral
\[
I_\chi(x):=-\frac{1}{2\pi i}\int_{2-i\infty}^{2+i\infty} \frac{L'}{L}(s,\chi)\frac{x^{s+1}}{s(s+1)}\,\dd s.
\]
Shifting the axis of integration arbitrarily far to the left, one gets for every $x>1$ the identity
\[
I_\chi(x) = \delta_\chi\frac{x^2}{2}
          - \sum_{\rho\in Z_\chi}\frac{x^{\rho+1}}{\rho(\rho+1)}
          - x r_\chi
          + r'_\chi
          + R_\chi(x)
\]
where $R_\chi(x)$ is defined in Lemma~\ref{lem:3.5} and $r_\chi$ and $r'_\chi$ are defined
in~\eqref{eq:3.2}. The shift is done in a way similar to~\cite[\S~6]{LagariasOdlyzko}, further simplified
by the fact that the integral is absolutely convergent on vertical lines. By~\eqref{eq:2.4},
Lemma~\ref{lem:3.2} and using $R_C$ as defined in Lemma~\ref{lem:3.5}, this gives
\begin{equation}
\frac{|G|}{|C|}\psi^{(1)}(C;x)
 = \MC I_\chi(x)
 = \frac{x^2}{2}
 - \sum_{\rho\in Z}\epsilon(\rho)\frac{x^{\rho+1}}{\rho(\rho+1)}
 - x \MC r_\chi
 + \MC r'_\chi
 + R_C(x)                                                        \label{eq:4.3}
\end{equation}
so that for any $h\neq 0$, one has
\begin{equation*}
\frac{|G|}{|C|}
\frac{\psi^{(1)}(C;x+h)-\psi^{(1)}(C;x)}{h}
= x
  + \frac{h}{2}
  - \sum_{\rho\in Z}\epsilon(\rho)\frac{(x+h)^{\rho+1}-x^{\rho+1}}{h\rho(\rho+1)}
  - \MC r_\chi
  + R'_C(\eta)
\end{equation*}
for a suitable $\eta$ in the interval between $x$ and $x+h$. By~\eqref{eq:2.1} we deduce for $h>0$:
\begin{align}
\frac{|G|}{|C|}\psi(C;x) - x
&\leq
  \frac{h}{2}
  + \sum_{\rho\in Z}\Big|\frac{(x+h)^{\rho+1}-x^{\rho+1}}{h\rho(\rho+1)}\Big|
  - \MC r_\chi
  + R'_C(\eta)                                                                \label{eq:4.4}
\intertext{and for $h<0$}
-\Big[\frac{|G|}{|C|}\psi(C;x) - x\Big]
&\leq
  - \frac{h}{2}
  + \sum_{\rho\in Z}\Big|\frac{(x+h)^{\rho+1}-x^{\rho+1}}{h\rho(\rho+1)}\Big|
  + \MC r_\chi
  - R'_C(\eta).                                                               \label{eq:4.5}
\end{align}
To get an upper bound for the sum of zeros we split its contribution into two parts: above and below $T$.
Moreover, in the lower range we isolate the contribution of $\sum_{|\gamma|\leq T}x^\rho/\rho$, which will
produce the main term. Thus,
\begin{align}
\sum_{\rho\in Z}\Big|\frac{(x+h)^{\rho+1}-x^{\rho+1}}{h\rho(\rho+1)}\Big|
&\leq \!\!\sum_{|\gamma|\leq T}\Big|\frac{x^{\rho}}{\rho}\Big|
    +\!\sum_{|\gamma|\leq T}\! \Big|\frac{(x+h)^{\rho+1}-x^{\rho+1}-h(\rho+1)x^\rho}{h\rho(\rho+1)}\Big|
    +\!\sum_{|\gamma|> T}\! \Big|\frac{(x+h)^{\rho+1}-x^{\rho+1}}{h\rho(\rho+1)}\Big|                    \notag\\
&\leq \sum_{|\gamma|\leq T}\frac{\sqrt{x}}{|\rho|}
    +\frac{|h|}{\sqrt{x}}\sum_{|\gamma|\leq T}|w_\rho|
    +\frac{x^{3/2}}{|h|}\Big(\Big(1+\frac{h}{x}\Big)^{3/2}+1\Big)\sum_{|\gamma|> T} \frac{1}{|\rho|^2}   \label{eq:4.6}
\end{align}
with
\[
w_\rho := \frac{\big(1+\frac{h}{x}\big)^{\rho+1}-1-(\rho+1)\frac{h}{x}}
               {\rho(\rho+1)\big(\frac{h}{x}\big)^2}.
\]
The technique we apply to bound~\eqref{eq:4.4} and~\eqref{eq:4.5} changes in some details. We thus
proceed separately for the two cases.\\
To prove~\eqref{eq:4.1} we bound the right hand side of~\eqref{eq:4.4}. Let $h>0$, then $|w_\rho|\leq
\frac{1}{2}$ from~\cite[Lemma 2.1]{GrenieMolteni3}, and~\eqref{eq:4.6} gives
\begin{align*}
\sum_{\rho\in Z}\Big|\frac{(x+h)^{\rho+1}-x^{\rho+1}}{h\rho(\rho+1)}\Big|
&\leq \sqrt{x}\sum_{|\gamma|\leq T}\frac{1}{|\rho|}
    +\frac{h}{2\sqrt{x}}N_\L(T)
    +\frac{x^{3/2}}{h}\Big(\Big(1+\frac{h}{x}\Big)^{3/2}+1\Big)\sum_{|\gamma|> T} \frac{1}{|\rho|^2}.
\end{align*}
By~\eqref{eq:3.14} we know that $N_\L(T)$ has order $T W_\L(T)$,
by~\eqref{eq:3.15} that $\sum_{|\gamma|> T}\frac{1}{|\rho|^2}$ has order $W_\L(T)/T$,
and
by~\eqref{eq:3.16} that $\sum_{|\gamma|\leq T}\frac{1}{|\rho|}$ has order $(\log T)W_\L(T)$.
The comparison of the second and the last term, hence, suggests to take $h\approx x/T$.
We set $h= 2x/T$. In this way we get:
\begin{align*}
\sum_{\rho\in Z}\Big|\frac{(x+h)^{\rho+1}-x^{\rho+1}}{h\rho(\rho+1)}\Big|
\leq& \sqrt{x}\sum_{|\gamma|\leq T}\frac{1}{|\rho|}
    +\frac{\sqrt{x}}{T}N_\L(T)
    +\frac{T\sqrt{x}}{2}\Big(\Big(1+\frac{2}{T}\Big)^{3/2}+1\Big)\sum_{|\gamma|> T} \frac{1}{|\rho|^2}.
\intertext{
Since $(1+\frac{2}{T})^{3/2}+1 \leq 2 + \frac{3}{T} + \frac{3}{2T^2}$ we conclude
}
\frac{1}{\sqrt{x}}
\sum_{\rho\in Z}\Big|\frac{(x+h)^{\rho+1}-x^{\rho+1}}{h\rho(\rho+1)}\Big|
\leq& \Big(\sum_{|\gamma|\leq T}\frac{1}{|\rho|}
           + \sum_{|\gamma|>T} \frac{2\pi}{|\rho|^2}
      \Big)
    +\frac{N_\L(T)}{T}
    +\Big(1 + \frac{3}{2T} + \frac{3}{4T^2} - \frac{2\pi}{T}\Big)
     \sum_{|\gamma|> T} \frac{T}{|\rho|^2}.
\end{align*}
Substituting~\eqref{eq:3.14},~\eqref{eq:3.15} and~\eqref{eq:3.16} in this equation, after some
rearrangements we get:
\enlargethispage{\baselineskip}
\begin{align}
\frac{1}{\sqrt{x}}
\sum_{\rho\in Z}\Big|\frac{(x+h)^{\rho+1}-x^{\rho+1}}{h\rho(\rho+1)}&\Big|
\leq \Big[\frac{1}{\pi}\log\Big(\frac{T}{2\pi}\Big)
          + 1.704
          + \frac{1.858}{T}
     \Big]\lDL                                                            \notag\\
&  + \Big[\frac{1}{2\pi}\log^2\Big(\frac{T}{2\pi}\Big)
          + \Big(\frac{2}{\pi}
                 + \frac{1.858}{T}
            \Big)\log\Big(\frac{T}{2\pi}\Big)
          - 1.633
          + \frac{7.729}{T}
     \Big]\nL                                                             \notag\\
&  + 7.834
   + \frac{3.779}{T}
   - \Big(0.406 + \frac{1}{4T}\Big)r_1
   - \ee_\nL.                                                             \label{eq:4.7}
\end{align}
The explicit formula for $R'_C$ in Lemma~\ref{lem:3.5} gives
\[
R'_C(\eta)\leq \log\eta-\Sg\log(\eta+1) + 0.256\nL\delta_C
\]
under the assumption that $x\geq 4$.
Using that and Lemma~\ref{lem:3.3},
\begin{equation}\label{eq:4.8}
 - \MC r_\chi
 + R'_C(\eta)
\leq \sum_{\rho\in Z}\frac{2}{|\rho(2-\rho)|}
    - \frac{\zeta'(2)}{\zeta(2)}\nL
    - \nL\delta_C
    + \Sg
    - \frac{5}{2}
    + (1-\Sg)\log x
    + 0.256\nL\delta_C
    + \frac{2}{T}.
\end{equation}
Following~\eqref{eq:4.4}, we sum~\eqref{eq:4.7} and \eqref{eq:4.8}, to get:
\begin{align}
\frac{|G|}{|C|}\psi(C;x) - x
\leq& \frac{x}{T}
   + \sqrt{x}\Big[\frac{1}{\pi}\log\Big(\frac{T}{2\pi}\Big)
          + 1.704
          + \frac{1.858}{T}
     \Big]\lDL                                                  \notag\\
&  + \sqrt{x}\Big[\frac{1}{2\pi}\log^2\Big(\frac{T}{2\pi}\Big)
          + \Big(\frac{2}{\pi}
                 + \frac{1.858}{T}
          \Big)\log\Big(\frac{T}{2\pi}\Big)
          - 1.633
          + \frac{7.729}{T}
     \Big]\nL                                                   \notag\\
&  + \sqrt{x}\Big[7.834
   + \frac{3.779}{T}
   - \Big(0.406 + \frac{1}{4T}\Big)r_1
   - \ee_\nL\Big]                                               \notag\\
&  + \sum_{\rho\in Z}\frac{2}{|\rho(2-\rho)|}
   - \frac{\zeta'(2)}{\zeta(2)}\nL
   - \nL\delta_C
   + \Sg
   - \frac{5}{2}
   + (1-\Sg)\log x
   + 0.256\nL\delta_C
   + \frac{2}{T}.                                               \label{eq:4.9}
\end{align}
Moreover, $|2-\rho| = |\rho+1|$ since we are assuming GRH. Thus, by Lemma~\ref{lem:3.10}
\[
\sum_{\rho\in Z} \frac{2}{|\rho(2-\rho)|} \leq 1.075\lDL - 2.071\nL + 10.776 - 0.527 r_1.
\]
The upper bound in~\eqref{eq:4.9} thus gives
\begin{align*}
\frac{|G|}{|C|}\psi(C&;x) - x
\leq
 \sqrt{x}\Big[\frac{1}{\pi}\log\Big(\frac{T}{2\pi}\Big)
          + 1.704
          + \frac{1.858}{T}
     \Big]\lDL                                                   \\
&  + \sqrt{x}\Big[\frac{1}{2\pi}\log^2\Big(\frac{T}{2\pi}\Big)
          + \Big(\frac{2}{\pi}
                 + \frac{1.858}{T}
          \Big)\log\Big(\frac{T}{2\pi}\Big)
          - 1.633
          + \frac{7.729}{T}
     \Big]\nL                                                    \\
&  + \sqrt{x}\Big[7.834
   + \frac{3.779}{T}
   - \Big(0.406 + \frac{1}{4T}\Big)r_1
   - \ee_\nL\Big]
 + 1.075\lDL
 - 2.071\nL
 + 10.776
 - 0.527 r_1                                                     \\
&+ 0.570\nL
 - 0.744\nL\delta_C
 + \Sg
 - \frac{5}{2}
 + (1-\Sg)\log x
 + \frac{x+2}{T}.
\end{align*}
This is the bound in~\eqref{eq:4.1}, once the definition of $L_a$ is considered.\\
%
To prove~\eqref{eq:4.2} we first bound the right hand side of~\eqref{eq:4.5}. In this case $h<0$, thus
$|w_\rho|\leq \frac{1}{2} + \frac{|h|}{6x}$ from~\cite[Lemma 2.1]{GrenieMolteni3}, so that~\eqref{eq:4.6}
gives
\begin{align*}
\sum_{\rho\in Z}\Big|\frac{(x+h)^{\rho+1}-x^{\rho+1}}{h\rho(\rho+1)}\Big|
&\leq \sqrt{x}\sum_{|\gamma|\leq T}\frac{1}{|\rho|}
    +\frac{|h|}{\sqrt{x}}\Big(\frac{1}{2}+\frac{|h|}{6x}\Big)N_\L(T)
    +\frac{x^{3/2}}{|h|}\Big(\Big(1+\frac{h}{x}\Big)^{3/2}+1\Big)\sum_{|\gamma|> T} \frac{1}{|\rho|^2}.
\end{align*}
Setting $h=-\frac{2x}{T}$, and estimating $(1+\frac{h}{x})^{3/2}+1 = (1 - \frac{2}{T})^{3/2}+1 \leq 2 -
\frac{3}{T} + \frac{20}{T^2}$ (valid as soon as $T\geq 2$),
we get
\begin{multline*}
\frac{1}{\sqrt{x}}
\sum_{\rho\in Z}\Big|\frac{(x+h)^{\rho+1}-x^{\rho+1}}{h\rho(\rho+1)}\Big|
\leq \Big(\sum_{|\gamma|\leq T}\frac{1}{|\rho|}+\sum_{|\gamma|>T} \frac{2\pi}{|\rho|^2}\Big)
    + \Big(1+\frac{2}{3T}\Big)\frac{N_\L(T)}{T}                                              \\
    + \Big(1 - \frac{3}{2T} + \frac{10}{T^2} - \frac{2\pi}{T}\Big)
              \sum_{|\gamma|> T} \frac{T}{|\rho|^2},
\end{multline*}
which with~\eqref{eq:3.14},~\eqref{eq:3.15} (which can be used because $1 - \frac{3}{2T} + \frac{10}{T^2}
- \frac{2\pi}{T}$ is positive for $T\geq 2\pi$) and~\eqref{eq:3.16} produces
\begin{align}
\frac{1}{\sqrt{x}}
  \sum_{\rho\in Z}\Big|\frac{(x+h)^{\rho+1}-x^{\rho+1}}{h\rho(\rho+1)}\Big|
&\leq \Big[\frac{1}{\pi}\log\Big(\frac{T}{2\pi}\Big)
           + 1.704
           + \frac{1.115}{T}
     \Big]\lDL                                                      \notag\\
  &+ \Big[\frac{1}{2\pi}\log^2\Big(\frac{T}{2\pi}\Big)
          + \Big(\frac{2}{\pi}
                 + \frac{1.115}{T}
                 - \frac{2.206}{T^2}
            \Big)\log\Big(\frac{T}{2\pi}\Big)
          - 1.633
          + \frac{6.562}{T}
     \Big]\nL                                                       \notag\\
  &+ 7.834
   + \frac{3.779}{T}
   - \frac{5.614}{T^2}
   - \Big(0.406 + \frac{1}{4T}\Big)r_1
   - \ee_\nL.                                                       \label{eq:4.10}
\end{align}
Then $-R'_C(\eta)\leq -\log\eta+\Sg\log(\eta+1)$ hence, using Lemma~\ref{lem:3.3},
\begin{equation}\label{eq:4.11}
 \MC r_\chi
 - R'_C(\eta)
\leq \sum_{\rho\in Z}\frac{2}{|\rho(2-\rho)|}
    + \nL\delta_C
    - \Sg
    + \frac{5}{2}
    + \Sg\log(x + 1)
    - \log\Big(x-\frac{2x}{T}\Big).
\end{equation}
Summing~\eqref{eq:4.10} and~\eqref{eq:4.11}, we get from~\eqref{eq:4.5}:
\begin{align*}
-\Big(\frac{|G|}{|C|}\psi(C;x) - x\Big)
\leq& \frac{x}{T}
   + \sqrt{x}\Big[\frac{1}{\pi}\log\Big(\frac{T}{2\pi}\Big)
                   + 1.704
                   + \frac{1.115}{T}
             \Big]\lDL                                         \\
&  + \sqrt{x}\Big[\frac{1}{2\pi}\log^2\Big(\frac{T}{2\pi}\Big)
                  + \Big(\frac{2}{\pi}
                         + \frac{1.115}{T}
                    \Big)\log\Big(\frac{T}{2\pi}\Big)
                  - 1.633
                  + \frac{6.562}{T}
             \Big]\nL                                          \\
&  + \sqrt{x}\Big[7.834
                  + \frac{3.779}{T}
                  - \frac{5.614}{T^2}
                  - \Big(0.406 + \frac{1}{4T}\Big)r_1
                  - \ee_\nL
             \Big]                                             \\
&   + \sum_{\rho\in Z}\frac{2}{|\rho(2-\rho)|}
    + \nL\delta_C
    - \Sg
    + \frac{5}{2}
    + (\Sg-1)\log x
    + \frac{\Sg}{x}
    - \log\Big(1-\frac{2}{T}\Big).
\end{align*}
Reorganizing as above we get
\begin{equation}\label{eq:4.12}
-\Big(\frac{|G|}{|C|}\psi(C;x) - x\Big) \leq L_a(x,T,\nL,\lDK) + \mathcal{A}
\end{equation}
with
\begin{align*}
\mathcal{A}
:=& 2(\Sg-1)\log x
   + \frac{\Sg}{x}
   - 2\Sg
   + 1.744\nL\delta_C
   + 5
   - 0.570\nL
   - \log\Big(1-\frac{2}{T}\Big)
   - \frac{2}{T}                                                                       \\
&  - \frac{\sqrt{x}}{T} \Big[0.743W_\L(T) + 1.167\nL + \frac{5.614}{T}\Big].
\end{align*}
%
We observe that, for $T\geq 2\pi$, we have $-\log(1-2/T)-2/T\leq 2.561/T^2\leq 5.614\sqrt{x}/T^2$,
%
%
and that $\Sg/x \leq 0.256\nL\delta_C + 0.125\nL$, under the assumption $x\geq 4$.
%
We then get
\begin{align}
\mathcal{A}
&\leq 2(\Sg-1)(\log x-1)
    - 0.445\nL
    + 3
    + 2 \nL\delta_C
    - \frac{\sqrt{x}}{T} (0.743W_\L(T)+1.167\nL) \notag\\
&= \DSG(x,T,\nL,\lDL).                           \label{eq:4.13}
\end{align}
%
By~\eqref{eq:1.2}, we have~\eqref{eq:4.2} from~\eqref{eq:4.12} and~\eqref{eq:4.13}.
\end{proof}

\section{Proof of Theorem~\ref{th:1.1}}\label{sec:5}
For $\L=\Q$, the theorem is weaker than Lowell Schoenfeld's result for $x\geq 59$, and true in the
range $[1,59]$ by explicit computation. We assume henceforth that $\L\neq \Q$, i.e. $\nL\geq 2$.\\
%
Since $\psi(C;x)\geq \psi_C(x)$, for the proof of the theorem it is sufficient to show that
\begin{align}
\frac{|G|}{|C|}\psi(C;x)-x
& \leq \sqrt{x}\Big[\Big(\frac{\log   x}{2\pi}+2\Big)\lDL
                       + \Big(\frac{\log^2 x}{8\pi}+2\Big)\nL\Big],  \label{eq:5.1}\\
-\Big(\frac{|G|}{|C|}\psi_C(x)-x\Big)
& \leq \sqrt{x}\Big[\Big(\frac{\log   x}{2\pi}+2\Big)\lDL
                       + \Big(\frac{\log^2 x}{8\pi}+2\Big)\nL\Big]   \label{eq:5.2}
\end{align}
hold $\forall x\geq 1$. Let then
\begin{align*}
B_a(x,T,n,\LC)   &:= \frac{L_a(x,T,n,\LC)}{n\sqrt{x}}
                     - \Big(\frac{\log x}{2\pi}+2\Big)\frac{\LC}{n}
                     - \Big(\frac{\log^2 x}{8\pi}+2\Big),                         \\
B_b(x,T,n,\LC,g) &:= B_a(x,T,n,\LC)
                     + \frac{\DSG(x,T,n,\LC)}{n\sqrt{x}}
                     + \frac{g}{p}\frac{\mathfrak{N}(\LC)}{n}\frac{\log x}{\sqrt{x}},
\end{align*}
where $g$ is an integer, $p$ is the smallest prime divisor of $g$ and $\mathfrak{N}(\lDL)$ is an upper
bound for $\omDLK$, as given by Lemma~\ref{lem:3.7}, that will be made explicit later.
%
%
To prove~\eqref{eq:5.1} it is sufficient to show that there is an $\bar{x}^+\geq 4$ such that it
is trivial for $x\in [1,\bar{x}^+]$ and that when $x\geq \bar{x}^+$, by~\eqref{eq:4.1}, there exists a
value of $T\geq 2\pi$ such that $B_a(x,T,\nL,\lDL)\leq 0$.
To prove~\eqref{eq:5.2} it is sufficient to show that there is an $\bar{x}^-\geq 4$ such that it
is trivial for $x\in [1,\bar{x}^-]$ and that when $x\geq \bar{x}^-$, by~\eqref{eq:4.2} and
Lemma~\ref{lem:3.6}, there exists a value of $T\geq 2\pi$ such that $B_b(x,T,\nL,\lDL)\leq 0$.\\
We assume, from now on, that $T=T(x):=c\sqrt{x}/\log x$ with $c:=5.2$. This ensures in particular that
$T\geq 2\pi$ for any $x>1$.
%

\subsection{Upper bound}
We first prove~\eqref{eq:5.1}.
\step{trivial bound}
We notice that $\psi(C;x)\leq\psi_\K(x)\leq\psi_\Q(x)\nK$. Hence, given that $\nL=|G|\nK$,
the bound~\eqref{eq:5.1} is true if
\[
\sqrt{x}\Big[\Big(\frac{\log x}{2\pi}+2\Big)\frac{\lDL}{\nL} + \Big(\frac{\log^2 x}{8\pi}+2\Big)\Big]
\geq \psi_\Q(x) - \frac{x}{\nL}.
\]
We will call this bound the trivial bound. We observe that $\psi_\Q$ is constant on the intervals
$[p^m,q^n)$ where $p^m$ and $q^n$ are consecutive prime powers, hence if the trivial bound is true in
$p^m$ it is true in the whole interval $[p^m,q^n)$. We check that the bound is true for $x<61$ if $\nL=4$
and for $x<71$ for any other value of $\nL\in[2,13]$ using the explicit lower bounds for $\lDL$
in~\cite{MegrezTables} and~\cite[Table~3]{OdlyzkoTables}. For $\nL\geq 14$,
$\frac{\lDL}{\nL}\geq 2.12$ as follows from entry $b=2.1$ in~\cite[Table~3]{OdlyzkoTables}. We this lower
bound, we check that the stronger bound without the $x/\nL$ term is true for $x<71$. This ensures that
it is true for $x<71$ and $\nL\geq 14$.\\
%
%
Hence~\eqref{eq:5.1} is a consequence of the trivial bound if either $\nL=4$ and $x< 61$ or $\nL\neq 4$ and
$x< 71$.
\step{function $B_a$ is decreasing in $\LC$}
We have
\begin{align*}
B_a(x,T(x),\nL,\LC)
=&  \Big[\frac{1}{\pi}\log\Big(\frac{c/(2\pi)}{\log x}\Big)
        - 0.296
        + \frac{1.858}{T}
        + \frac{1.075}{\sqrt{x}}
   \Big]\frac{\LC}{\nL}                                                         \\
&+      \frac{1}{2\pi}\log^2\Big(\frac{T}{2\pi}\Big)
        - \frac{1}{8\pi}\log^2 x
        + \Big(\frac{2}{\pi} + \frac{1.858}{T}\Big)\log\Big(\frac{T}{2\pi}\Big)
        - 3.633
        + \frac{7.729}{T}
        - \frac{1.501}{\sqrt{x}}                                                \\
&+ \frac{1}{\nL\sqrt{x}}\Big[
               \frac{x+2}{T}
               + (1-\Sg)\log x
               + \Sg
               + 8.276
               - 0.744\nL\delta_C
               - 0.527 r_1
   \Big]                                                                        \\
&+ \frac{1}{\nL}\Big[7.834+\frac{3.779}{T}-\Big(0.406+\frac{1}{4T}\Big)r_1-\ee_\nL\Big].
\end{align*}
Since $T(x)$ is an increasing function of $x\geq e^2$, $\frac{\partial B_a}{\partial\LC}$ is decreasing
with $x$. As $\frac{\partial B_a}{\partial\LC}(61,T(61))\leq 0$, we have that $\frac{\partial
B_a}{\partial\LC}\leq 0$ for any $x\geq 61$.
%
%
\step{function $B_a$ is decreasing in $x$}
We have
\begin{align*}
\frac{\partial B_a}{\partial x}(x,T(x),\nL,\LC)
\leq& \frac{-\log 3}{2}\Big[\frac{1}{\pi x\log x}
                      + \frac{1.858T'}{T^2}
                      + \frac{1.075}{2x\sqrt{x}}
                 \Big]
 + \frac{T'}{T}\Big(\frac{1}{\pi} - \frac{1.858}{T}\Big)\log\Big(\frac{T}{2\pi}\Big) \\
&- \frac{\log x}{4\pi x}
 + \frac{2T'}{\pi T}
 - \frac{5.621T'}{T^2}
%
 + \frac{\log x + 0.772}{2x\sqrt{x}}
%
%
 + \frac{1}{c\nL x}
 - \frac{4.138}{\nL x^{3/2}}
\end{align*}
where we have removed a few terms whose decreasing behaviour is evident, and used the facts that
$\LC/\nL\geq \frac{1}{2}\log 3$, $\delta_C\leq 1$, $\Sg\leq \nL$ and $r_1\leq \nL$. Since $\nL\geq 2$, we
bound the last two terms by $\max(0,1/(cx)-4.138x^{-3/2})/2$ and the resulting function is an elementary
one variable function which is negative for $x\geq 61$.
%
%
\step{estimates for $\nL\geq 4$}
For $\nL\geq 4$, we have $\lDL\geq \nL$ (this is true for all number fields except $\Q$ and the four
quadratic fields with $\DL\leq 7$). Given that $B_a$ is a decreasing function of $\LC$ for $x\geq 61$, we
have
\[
B_a(x,T(x),\nL,\lDL)\leq B_a(x,T(x),\nL,\nL)
\]
as soon as $\nL\geq 4$ and $x\geq 61$.\\
Since $\delta_C\geq 0$, $r_1\geq 0$, $\Sg\geq 0$ and $\ee_\nL\geq 0$, we have
\begin{align*}
B_a(x,T(x),\nL,\nL)
\leq&      \frac{1}{\pi}\log\Big(\frac{c/(2\pi)}{\log x}\Big)
           - 0.296
           + \frac{1.858}{T}
           + \frac{1.075}{\sqrt{x}}                                       \\
   &+      \frac{1}{2\pi}\log^2\!\Big(\frac{T}{2\pi}\Big)
           - \frac{1}{8\pi}\log^2 x
           + \Big(\frac{2}{\pi}
                  + \frac{1.858}{T}
             \Big)\log\Big(\frac{T}{2\pi}\Big)
           - 3.633
           + \frac{7.729}{T}
           - \frac{1.501}{\sqrt{x}}                                       \\
   &+ \frac{1}{\nL\sqrt{x}}\Big[
                  \frac{x+2}{T}
                  + \log x
                  + 8.276
      \Big]
    + \frac{1}{\nL}\Big[7.834 + \frac{3.779}{T}\Big].
\end{align*}
This upper bound is decreasing in $\nL$ because $\nL$ only appears as the denominator of a fraction with
positive numerator. Since $B_a(61,T(61),4,4)<0$, the decreasing behaviour of $B_a$ in $x$, $n$ and
$\LC$ proves that $B_a(x,T(x),\nL,\lDL)<0$ if $\nL\geq 4$ and $x\geq 61$.
%
%
With the trivial bound in Step~1, we see that $B_a(x,T(x),\nL,\lDL)<0$ if $\nL\geq 4$ and $x\geq 1$.
\step{estimates for $\nL=3$, $r_1=3$}
In this case $\DL\geq 49$ and $B_a(71,T(71),3,\log 49)<0$ (where we use, as above, that
$\delta_C\geq 0$ and $\Sg\geq 0$) which, including the trivial bound, concludes the proof.
%
%
\step{estimates for $\nL=3$, $r_1=1$}
In this case $\DL\geq 23$ and we necessarily have $\L=\K$, hence $\delta_C=1$ and $\Sg=(\nL+r_1)/2=2$.
Since $B_a(71,T(71),3,\log 23)<0$,
the proof is complete for $\nL= 3$.
\step{estimates for $\nL=2$, large $\DL$ or large $x$}
We observe that the trivial bound extends to $x < 607$ when $\DL\geq 300$.
%
%
As above the worst case is for $\delta_C=0$ and $r_1=0$ and in that case $\Sg=1$. We have
$B_a(607,T(607),2,\log 300)<0$, which means that the case where $\nL=2$, $\DL\geq 300$ is proved.\\
%
%
Besides, we observe that also $B_a(10^5,T(10^5),2,\log 3)<0$, keeping the worst case $\delta_C=0$,
$r_1=0$ and $\Sg=1$, hence~\eqref{eq:5.1} for $\nL=2$ is proved also for $x\geq 10^5$.
%
%
Hence~\eqref{eq:5.1} is proved for $\nL=2$ if either $\DL\geq 300$ or $x\geq 10^5$.
\step{estimates for $\nL=2$, small $\DL$ and small $x$}
For the remaining quadratic fields $\L$ the proof will be made together with the lower bound.
\subsection{Lower bound}
We now turn to~\eqref{eq:5.2}.\\
Lemma~\ref{lem:3.7}\eqref{eq:lem3.7_4} shows that $\omDLK\leq\lDL/(\llDL-\log\nK-1.1714)$ when
$\lDL>e^{1.1714}\,\nK$. To get an easier estimate we use line $b=4.1$ of Table~3 in~\cite{OdlyzkoTables},
producing the lower bound
\begin{align*}
\llDL-\log\nK-1.1714
&\geq \log(\nL\log 25.585-28.36)-\log\nK-1.1714             \\
&=    \log\Big(|G|\log 25.585-\frac{28.36}{\nK}\Big)-1.1714
 \geq \log(|G|-8.79).
\end{align*}
%
%
Moreover, Lemma~\ref{lem:3.7}\eqref{eq:lem3.7_3} implies that $\omDLK\leq 0.4+\lDL/\log 22$ if $|G|$ is not
prime -- where the $0.4$ has been added to handle the exceptions. We thus define
\[
\mathfrak{N}(\LC)
:= \begin{cases}
        0                   & \text{if }|G|=1,                                   \\
        \LC/\log(|G|-8.79)  & \text{if }|G|\geq 32,                              \\
        \LC/\log 4          & \text{if $|G|$ is a prime $\leq 31$ and $\neq 3$}, \\
        \LC/\log 49         & \text{if }|G|=3,                                   \\
        0.4 + \LC/\log 22   & \text{otherwise}.
   \end{cases}
\]
%
In this way, from Lemma~\ref{lem:3.7} we have $\omDLK\leq \mathfrak{N}(\lDL)$.\\
Before starting the proof, we observe that if $\K=\L$, then $\mathfrak{N}(\LC)=0$. Thus, when we are able
to prove that $B_b \leq 0$ for suitable $x$, $T$ (and a certain value for the parameters $r_1$ and $\Sg$)
under the assumption that $\K\neq \L$, then with the same values for $x$ and $T$, we have $B_b \leq 0$
also for $\K=\L$ (and the same value for $r_1$ and $\Sg$).
\step{trivial bound}
Bound~\eqref{eq:5.2} is satisfied if
\[
\Big(\frac{\log x}{2\pi}+2\Big)\lDL + \Big(\frac{\log^2 x}{8\pi}+2\Big)\nL \geq \sqrt{x}
\]
because in this case it is weaker than the trivial bound $\psi_C(x)\geq 0$. Since for $\nL\geq 3$ we have
$\lDL\geq \nL$ we see that this is true if $x\leq 16 \nLsq$. This extends to $\nL=2$ by direct
computation.\\
%
%
For the end of this subsection, we will assume $x\geq 16\nLsq$ (and hence $x\geq 16|G|^2$ and $x\geq
64$).
\step{function $B_b$ is decreasing in $\LC$}
We have
\begin{align*}
B_b(x,T(x),\nL,\LC,|G|)\!\!
=& \Big[\frac{1}{\pi}\log\Big(\frac{c/(2\pi)}{\log x}\Big)
        - 0.296
        + \frac{1.115}{T}
        + \frac{1.075}{\sqrt{x}}
   \Big]\frac{\LC}{\nL}
 + \frac{|G|}{p}\frac{\mathfrak{N}(\LC)}{\nL}\frac{\log x}{\sqrt{x}}                                \\
&+      \frac{1}{2\pi}\log^2\Big(\frac{T}{2\pi}\Big)
        - \frac{1}{8\pi}\log^2 x
        + \Big(\frac{2}{\pi} + \frac{1.115}{T}\Big)\log\Big(\frac{T}{2\pi}\Big)                     \\
&       - 3.633
        + \frac{6.562}{T}
        - \frac{1.946}{\sqrt{x}}
%
 + \frac{1}{\nL}\Big[7.834-\Big(0.406+\frac{1}{4T}\Big)r_1-\ee_\nL+\frac{3.779}{T}\Big]             \\
&+ \frac{1}{\nL\sqrt{x}}\Big[
        \frac{x+2}{T}
        + (\Sg-1)\log x
        - \Sg
        + 13.276
        + 1.256\nL\delta_C
        - 0.527 r_1
   \Big].                                                                                           \\
\end{align*}
We observe that the derivative $\mathfrak{N}'$ is a constant depending only on $|G|$. Moreover, since
$x\geq 16\nLsq\geq 16|G|^2$,
\[
\frac{\partial}{\partial\LC}\Big[\frac{|G|}{p}\mathfrak{N}(\LC)\frac{\log x}{\sqrt{x}}\Big]
  =  \frac{|G|\mathfrak{N}'\log x}{p\sqrt{x}}
\leq \frac{\mathfrak{N}'\log(4|G|)}{2p}.
\]
By computing the values for $2\leq |G|\leq 32$, and using the lower bound $x\geq 16|G|^2$, we observe that
\[
\frac{1.075}{\sqrt{x}} + \frac{\mathfrak{N}'\log(4|G|)}{2p}
\leq 0.51.
\]
The conclusion holds also for any $|G|>32$ because
\[
\frac{\mathfrak{N}'\log(4|G|)}{2p}
\leq \frac{\log(4|G|)}{4\log(|G|-8.79)}
\]
which decreases in $|G|$.
We thus get
\[
\nL\frac{\partial B_b}{\partial \LC}\leq
\frac{1}{\pi}\log\Big(\frac{c/(2\pi)}{\log x}\Big)
- 0.296
+ \frac{1.115}{T}
+ 0.51
\]
which is negative because $x\geq 64$ hence $T\geq 10$.
%
%
\step{function $B_b$ is decreasing in $x$}
We have
\begin{align*}
\frac{\partial B_b}{\partial x}(x,T(x),\nL,\LC,|G|)
\leq& \frac{-\log 3}{2}\Big[\frac{1}{\pi x\log x}
                      + \frac{1.115T'}{T}
                      + \frac{1.075}{2x\sqrt{x}}
                 \Big]
 + \frac{T'}{T}\Big(\frac{1}{\pi} - \frac{1.115}{T}\Big)\log\Big(\frac{T}{2\pi}\Big) \\
&- \frac{\log x}{4\pi x}
 + \frac{2T'}{\pi T}
 - \frac{5.197T'}{T^2}
%
%
 + \frac{2.473}{2x\sqrt{x}}
%
%
 + \frac{1}{c\nL x}
 + \frac{\log x - 2}{2\nL x\sqrt{x}}
\end{align*}
which is negative as well for $x\geq 64$.
%
%
\step{estimates for $\nL\geq 4$}
We have $\lDL\geq \nL$. Given that $B_b$ is a decreasing function of $\LC$ for $x\geq 64$, we have
\[
B_b(x,T(x),\nL,\lDL,|G|)\leq B_b(x,T(x),\nL,\nL,|G|)
\]
as soon as $x\geq 64$. We know that $\Sg\leq (\nL+r_1)/2$; introducing this bound in $B_b$, the term
depending on $r_1$ in $B_b$ becomes
\[
\frac{r_1}{\nL\sqrt{x}}\Big(\frac{1}{2}(\log x - 1) - 0.527 - \Big(0.406 + \frac{1}{4T}\Big)\sqrt{x}\Big)
\]
which is $\leq 0$ for every $x$.
%
%
Its larger value is therefore reached for $r_1=0$. Once the bound $\delta_C\leq 1$ is also considered,
we get the upper bound
\begin{align*}
B_b(x,T,\nL,\nL,|G|)
\leq&      \frac{1}{\pi}\log\Big(\frac{c/(2\pi)}{\log x}\Big)
           - 0.296
           + \frac{1.115}{T}
           + \frac{1.075}{\sqrt{x}}
           + \frac{|G|}{p}\frac{\mathfrak{N}(\nL)}{\nL}\frac{\log x}{\sqrt{x}} \\
   &+      \frac{1}{2\pi}\log^2\Big(\frac{T}{2\pi}\Big)
           - \frac{1}{8\pi}\log^2 x
           + \Big(\frac{2}{\pi}
                  + \frac{1.115}{T}\Big)\log\Big(\frac{T}{2\pi}\Big)           \\
   &       - 3.633
           + \frac{6.562}{T}
           + \frac{\log x-2.380}{2\sqrt{x}}
           + \frac{1}{\nL\sqrt{x}}\Big[\frac{x+2}{T}
                                     - \log x
                                     + 13.276
                                  \Big]                                                             \\
   &       + \frac{1}{\nL}\Big[7.834 + \frac{3.779}{T}\Big].
\end{align*}
Once again this is decreasing in $\nL$, as long as $|G|/p$ remains constant and $\mathfrak{N}$ does not
change form, since $7.834\sqrt{x}-\log x>0$.
%
%
We check that $B_b$ is negative in the proper range of its arguments by checking that this upper bound is
negative, too. Doing this, we can restrict the test to the cases with $|G|\geq 2$: in fact,
$\frac{|G|}{p}\frac{\mathfrak{N}(\nL)}{\nL}\frac{\log x}{\sqrt{x}}$ is the unique term depending on $|G|$
appearing there, and it is zero when $|G|=1$. Moreover, for each $|G|$, we only need to check whether the
right hand side with $x=16 \nLsq$, $T=T(16 \nLsq)$ is negative when $\nL=|G|$ (if $|G|\geq 4$) or when
$\nL=2|G|$ (if $|G|=2$ or $3$).\\
If $|G|\geq 32$, then $\nL=|G|$ and
\[
\frac{|G|}{p}\frac{\mathfrak{N}(\nL)}{\nL}\frac{\log(16\nLsq)}{\sqrt{\smash[b]{16\nLsq}}}
  =  \frac{\log(4|G|)}{2p\log(|G|-8.79)}
\leq \frac{\log(4|G|)}{4\log(|G|-8.79)}
\]
which is decreasing in $|G|$, so, we just need to test the value for $\nL=|G|=32$.\\
If $|G|\leq 31$ is not prime, we need to check for $|G|/p\in\{2,\ldots,15\}$, but from the decreasing
argument (now in $p$ with fixed $|G|/p$) we only need to check the case $p=2$, i.e. $|G|$ even in
$[4,30]$.\\
If $|G|\leq 31$ is prime (but different from $3$) we have
\[
\frac{|G|}{p}\frac{\mathfrak{N}(\nL)}{\nL}\frac{\log(16\nLsq)}{\sqrt{\smash[b]{16\nLsq}}}
  =  \frac{\log(4\nL)}{2\nL\log 4},
\]
which decreases in $\nL$. Thus we just need to check the case $\nL=4$, and hence $|G|=2$.\\
If $|G|=3$, then $\nL=6$ and
\[
\frac{|G|}{p}\frac{\mathfrak{N}(\nL)}{\nL}\frac{\log(16\nLsq)}{\sqrt{\smash[b]{16\nLsq}}}
  =  \frac{\log(4\nL)}{2\nL\log 49},
\]
which is smaller than what we got previously for the case $|G|=6$.\\
In total we have sixteen cases: $\nL=|G|=32$, $\nL=|G|$ even in $[4,30]$ and $\nL=4$ with $|G|=2$.
All sixteen values are negative.
%
%
We have covered all cases for $|G|/p$ and $\mathfrak{N}$ hence, together with the trivial bound, this
proves the lower bound for $\nL\geq 4$.
\step{estimates for $\nL=3$}
We have $\DL\geq 23$, $\delta_C\leq 1$. As for the previous case, we estimate $\Sg$ with $(\nL+r_1)/2$
and the emerging term depending on $r_1$ with its largest value, which now corresponds to $r_1=1$
(because for $\nL=3$ the unique admissible values for $r_1$ are $1$ and $3$). This produces the bound
\begin{align*}
B_b(x,T(x),3,\log 23&,|G|)\!\!
\leq \Big[\frac{1}{\pi}\log\!\Big(\!\frac{c/(2\pi)}{\log x}\!\Big)
        - 0.296
        + \frac{1.115}{T}
        + \frac{1.075}{\sqrt{x}}
     \Big]\frac{\log 23}{3}
 + \frac{\log 23}{3\log 49}\frac{\log x}{\sqrt{x}}                                                  \\
&+ \frac{1}{2\pi}\log^2\Big(\frac{T}{2\pi}\Big)
 - \frac{1}{8\pi}\log^2 x
 + \Big(\frac{2}{\pi} + \frac{1.115}{T}\Big)\log\Big(\frac{T}{2\pi}\Big) \\
&- 3.633
        + \frac{6.562}{T}
        - \frac{1.946}{\sqrt{x}}
 + \frac{1}{3}\Big[7.115 + \frac{3.529}{T}\Big]
 + \frac{1}{3\sqrt{x}}\Big[\frac{x+2}{T}
                            + \log x
                            + 14.517
                      \Big],
%
\end{align*}
which is negative for $x=16\nLsq = 16\cdot 9$ and $T=T(16\cdot 9)$.
%
%
This completes the proof of the claim for $\nL=3$.
\step{estimates for $\nL=2$, large $\DL$ or large $x$}
The worst case happens when $\delta_C=1$, $|G|=2$, $\Sg=1+r_1/2$ and $r_1=0$. For $\DL\geq 300$, we
observe that the trivial bound extends to $x\leq 598$ and that
%
%
$B_b(598,T(598),2,\log 300,2)<0$ if $r_1=0$. This means that the case where $\DL\geq 300$ is proved.
%
%
We observe that $B_b(10^5,T(10^5),2,\log 3,2)<0$,
%
%
hence the claim is proved for $x\geq 10^5$.
\step{estimates for $\nL=2$, small $\DL$ and small $x$}
For the remaining fields $\L$, which are quadratic with $\DL< 300$, let $x_1(\L)\geq 61$ be such that
$B_a(x_1(\L),T(x_1(\L)),\nL,\lDL)<0$ (with $\delta_C=0$) and $B_b(x_1(\L),T(x_1(\L)),\nL,\lDL,2)<0$ (with
$\delta_C=1$), where we use the true value of $\omDLK$. As we have seen, for all fields $x_1(\L)\leq
10^5$. To complete the proof of Theorem~\ref{th:1.1} we have built a program that checks for each integer
$x\in[1,x_1(\L)]$ that
\begin{align*}
-\mathcal{B} +1 &\leq \psi_\L(x)-x \leq \mathcal{B},                   \\
-\mathcal{B} +1 &\leq 2\psi_C(x)-x \leq 2\psi(C;x)-x \leq \mathcal{B},
\end{align*}
where
\[
\mathcal{B} := \sqrt{x}\Big[\Big(\frac{\log x}{2\pi}+2\Big)\lDL + 2\Big(\frac{\log^2 x}{8\pi}+2\Big)\Big].
\]

\section{Proof of Corollary~\ref{cor:1.2}}\label{sec:6}
The bounds stated in the corollary are certainly true as soon as
\begin{equation*}
\sqrt{x}\Big[\Big(\frac{     1}{2\pi} + \frac{3}{\log x}\Big)\frac{\lDL}{\nL}
           + \Big(\frac{\log x}{8\pi} + \frac{1}{4\pi} + \frac{6}{\log x}\Big)
        \Big]
\geq \max\Big(\int_2^x\frac{\dd u}{\log u},\pi(x) - \frac{1}{\nL}\int_2^x\frac{\dd u}{\log u}\Big),
\end{equation*}
because in this case the conclusion is weaker than the elementary bound $0\leq \pi_C(x) \leq \pi(C;x)
\leq \pi(x)\nK$. The first inequality holds when $x\in[2,193)$, because $\frac{1}{\nL}\lDL\geq
\frac{1}{2}\log 3$, and
\[
\sqrt{x}\Big[\Big(\frac{     1}{2\pi} + \frac{3}{\log x}\Big)\frac{1}{2}\log 3
           + \Big(\frac{\log x}{8\pi} + \frac{1}{4\pi} + \frac{6}{\log x}\Big)
        \Big]
\geq \int_2^x\frac{\dd u}{\log u}
\]
holds in this range.
The second inequality
\[
\sqrt{x}\Big[\Big(\frac{     1}{2\pi} + \frac{3}{\log x}\Big)\frac{\lDL}{\nL}
           + \Big(\frac{\log x}{8\pi} + \frac{1}{4\pi} + \frac{6}{\log x}\Big)
        \Big]
\geq \pi(x) - \frac{1}{\nL}\int_2^x\frac{\dd u}{\log u}
\]
is checked for $x\in[2,193)$ by testing it for each $\nL\leq 20$ (using the lower bound for $\lDL$ as
follows from Odlyzko's tables for each degree). The case $\nL=20$ is checked in the stronger version
where $ - \frac{1}{\nL}\int_2^x\frac{\dd u}{\log u}$ is removed, so that its validity implies the
validity also for all $\nL\geq 20$.
%
%
In this way the corollary is fully proved up to $193$.\\
Let
\begin{align*}
\vartheta(C;x) &:= \sum_{\substack{\p\\\Norm\p\leq x}}\theta(C;\p) \log\Norm\p.
\end{align*}
Then by partial summation
\begin{align*}
\frac{|G|}{|C|}\pi(C;x) - \int_2^x \frac{\!\dd u}{\log u}
 &=   \frac{|G|}{|C|}\pi(C;73)
  -   \frac{\frac{|G|}{|C|}\vartheta(C;73)}{\log 73}
  +   \frac{73}{\log 73}
  -   \int_2^{73} \frac{\!\dd u}{\log u}                                  \\
 &\quad
  +   \frac{\frac{|G|}{|C|}\vartheta(C;x)-x}{\log x}
  +   \int_{73}^x \frac{\frac{|G|}{|C|}\vartheta(C;u)-u}{u\log^2 u}\dd u.
\end{align*}
Assuming $x\geq 193$, we have
%
%
\begin{equation*}
0\leq\pi(C;73)-\frac{\vartheta(C;73)}{\log 73}
 \leq \sum_{\!\Norm\p\leq 73\!}\Big(1-\frac{\log \Norm\p}{\log 73}\Big)
 \leq \sum_{\substack{p\leq 73\\\!\! p \text{ prime}\!\!}}\Big(1-\frac{\log p}{\log 73}\Big)\nK
 \leq 5.65\nK
 \leq 2.15\frac{\sqrt{x}}{\log x}\nK,
\end{equation*}
%
%
\[
0\leq \int_2^{73}\frac{\!\dd u}{\log u}-\frac{73}{\log 73}
 \leq 6.1
 \leq 1.16\frac{\sqrt{x}}{\log x}\nL,
\]
%
%
and
\[
\forall x\geq 1,\qquad
0\leq \psi(C;x)-\vartheta(C;x)\leq \psi_\K(x)-\vartheta_\K(x)
\leq 1.43\sqrt{x}\nK
\]
by~\cite[Th.~13]{RosserSchoenfeld}. We deduce that
\begin{align*}
\Big|\frac{|G|}{|C|}\pi(C;x) - \int_2^x \frac{\!\dd u}{\log u}\Big|
\leq&  \frac{\big|\frac{|G|}{|C|}\psi(C;x)-x\big| + 2.59\sqrt{x}\nL}{\log x}
     + \int_2^x \frac{\big|\frac{|G|}{|C|}\psi(C;u)-u\big|+ 1.43\sqrt{u}\nL}{u\log^2 u}\dd u \\
%
%
\leq&  \sqrt{x}\Big[\Big(\frac{     1}{2\pi} + \frac{2}{\log x}\Big)\lDL
                   + \Big(\frac{\log x}{8\pi} + \frac{4.59}{\log x}\Big)\nL
                \Big]                                                                        \\
 &+   \int_{73}^x \frac{\big(\frac{\log u}{2\pi} + 2\big)\lDL
                      + \big(\frac{\log^2 u}{8\pi} + 3.43\big)\nL}{\sqrt{u}\log^2 u}\dd u.
\intertext{Since $\int_{73}^x \frac{\frac{\log u}{2\pi} + 2}{\sqrt{u}\log^2 u}\dd u \leq
\frac{\sqrt{x}}{\log x}$,
%
%
and $\int_{73}^x \frac{d\! u}{\sqrt{u}\log^2 u} \leq 0.33\frac{\sqrt{x}}{\log x}$ (for $x\geq 193$),
%
%
we get}
\leq& \sqrt{x}\Big[\Big(\frac{     1}{2\pi} + \frac{3}{\log x}\Big)\lDL
                 + \Big(\frac{\log x}{8\pi} + \frac{1}{4\pi} + \frac{6}{\log x}\Big)\nL
              \Big],
\end{align*}
which concludes the proof of the claim for $\pi(C;x)$. For $\pi_C(x)$ the argument is the same.

\appendix
\section{Number of zeros}
Trudgian~\cite{TrudgianIII} showed how to take advantage of both Backlund's and Rosser's approaches to
produce good explicit bounds for the function $N(T)$ counting non-trivial zeros $\rho$ with
$|\Imm\rho|\leq T$ for Dirichlet and Dedekind $L$-functions. Note that, contrary to the rest of this
paper, Trudgian's approach doest not require to assume any form of the Riemann Hypothesis. Studying his
paper we have found some
possible improvements in the way some terms are bounded. We have also noted that the original paper does
not isolate the role of a special constant (the analogue of the constant $-7/8$ appearing for Riemann's
zeta in~\cite[Ch.~15, (1)]{Davenport1}). However, isolating this term allows to formulate the bound with
smaller constants, and this is very useful when sums on zeros of type $\sum_{|\Imm \rho|\geq a} f(\rho)$
with $a>0$ are estimated via partial summation, because in this case that term does not contribute and
only the smaller constants appear. This is very important for our application, since we need to take
advantage of every possible method to improve the constants, in order to reduce the set of explicit
computations which are needed to prove Theorem~\ref{th:1.1}.\\
Moreover, we have also noticed that essentially the same strategy can be applied to study the zeros of
all Hecke's $L$-functions of finite order Gr\"{o}{\ss}en\-cha\-rak\-ter, thus we have formulated the results
for this more general set, for possible future reference.\\
We stress once again that the main strategy for this computation has to be credited to Trudgian, our
contribution being limited to the points cited above.
\smallskip

Let $\E$ be a number field. %
Let $\chi$ be a Hecke Gr\"{o}{\ss}encharakter of $\E$ which is primitive and of finite order. %
Let $\cond(\chi)$ denote the conductor of $\chi$ and set $Q(\chi) = \Delta_\E \Norm_{\E/\Q}(\cond(\chi))$. %
Let $\delta_\chi$ be $1$ if $\chi$ is trivial and $0$ otherwise. %
Let $N(T,\chi)$ be the number (multiplicity included) of non-trivial zeros $\rho$ (i.e. with $\Ree \rho
\in (0,1)$) with $|\Imm\rho|\leq T$ for $L(s,\chi)$.
\begin{theorem}\label{th:A1}Unconditionally,
\begin{equation*}
\Big|
N(T,\chi) - \frac{T}{\pi}\log\Big[Q(\chi)\Big(\frac{T}{2\pi e}\Big)^\nE\Big] - 2\delta_\chi + \frac{a_\chi-b_\chi}{4}
\Big|
\leq D_1(\log Q(\chi) + \nE \log T) + D'_2\nE + \delta_\chi D'_3
\end{equation*}
when $T\geq T_0$, for $T_0$, $D_1$, $D'_2$ and $D'_3$ as in Table~\ref{tab:A1}.
\end{theorem}
If $\chi$ is the trivial character, then $\E=\L$ and $N(T,\chi)=N_\L(T)$ is the number of non-trivial
zeros of $\zeta_\L$ with imaginary part in $[-T,T]$. In that case $Q(\chi)=\DL$ and $a_\chi-b_\chi=r_1$.
If one want to compare this result with the analogue contained in~\cite[Theorem~2]{TrudgianIII} one has
to take note of the extra term $-2+\frac{1}{4}r_1$ that we have put in evidence (as for Riemann's
zeta in~\cite[Ch.~15, (1)]{Davenport1}).

\begin{fixedtab}
\caption{Parameters for Theorem~\ref{th:A1}}\label{tab:A1}
\begin{tabular}{rrrrrrr}
  \toprule
   \mltc{1}{c|}{}  &         \mltc{2}{c|}{$T_0=1$}          &        \mltc{2}{c|}{$T_0=2\pi$}        &         \mltc{2}{c}{$T_0=10$}         \\
\mltc{1}{c|}{$D_1$}&\mltc{1}{c}{$D'_2$}&\mltc{1}{c|}{$D'_3$}&\mltc{1}{c}{$D'_2$}&\mltc{1}{c|}{$D'_3$}&\mltc{1}{c}{$D'_2$}&\mltc{1}{c}{$D'_3$}\\
  \midrule
0.230 & 16.577 & 1.330 & 16.032 & 0.033 & 16.004 & 0.014 \\
0.247 &  8.180 & 1.435 &  7.614 & 0.083 &  7.585 & 0.062 \\
0.265 &  6.416 & 1.515 &  5.834 & 0.150 &  5.805 & 0.129 \\
0.282 &  5.409 & 1.598 &  4.812 & 0.213 &  4.783 & 0.192 \\
0.299 &  4.696 & 1.699 &  4.083 & 0.275 &  4.053 & 0.254 \\
0.316 &  4.158 & 1.814 &  3.526 & 0.335 &  3.495 & 0.313 \\
0.333 &  3.735 & 1.961 &  3.082 & 0.400 &  3.050 & 0.371 \\
0.350 &  3.425 & 2.185 &  2.731 & 0.429 &  2.698 & 0.402 \\
0.367 &  3.206 & 2.426 &  2.467 & 0.453 &  2.432 & 0.423 \\
0.384 &  3.043 & 2.687 &  2.257 & 0.478 &  2.221 & 0.444 \\
0.401 &  2.918 & 2.966 &  2.083 & 0.503 &  2.044 & 0.465 \\
0.460 &  2.666 & 4.082 &  1.645 & 0.593 &  1.598 & 0.540 \\
  \bottomrule
\end{tabular}
\end{fixedtab}
\begin{proof}
We first suppose that $\chi$ is non-trivial.
Let $\sigma_1\in(1,2)$ and let $\Rp$ be the rectangle with vertices $\sigma_1\pm iT$ and $1-\sigma_1\pm
iT$, positively oriented. We furthermore assume that $T$ is not the imaginary part of any zero of
$L(s,\chi)$. The conclusion for the missing $T$'s follows because $N(T,\chi)$ is upper-continuous and
all other functions are continuous. Cauchy's argument principle shows that
\[
2\pi N(T,\chi) = \Delta_\Rp \arg\xi(s,\chi),
\]
where $\Delta_\Rp\arg\xi(s,\chi)$ is the variation of the argument of $\xi(s,\chi)$ along $\Rp$. The
functional equation shows that the variation of the argument we have in the left half-rectangle equals the
variation in the right half-rectangle. Hence
\[
\pi N(T,\chi) = \Delta_\Cp \arg\xi(s,\chi)
\]
where $\Cp$ is the path $1/2-iT \to \sigma_1-iT \to \sigma_1+iT \to 1/2 + iT$ and $\Delta_\Cp$ is the
variation along $\Cp$. Hence
\begin{align*}
\pi N(T,&\chi)
= \Delta_\Cp \arg (Q(\chi)^{\frac{s}{2}})
 +\Delta_\Cp \arg \Gamma_\chi(s)
 +\Delta_\Cp \arg L(s,\chi)\\
&= \Delta_\Cp \arg (Q(\chi)^{\frac{s}{2}})
 + a_\chi\Delta_\Cp \arg\!\Big(\pi^{-\frac{s}{2}}\Gamma\Big(\frac{s}{2}\Big)\!\Big)
 + b_\chi\Delta_\Cp \arg\!\Big(\pi^{-\frac{s+1}{2}} \Gamma\Big(\frac{s+1}{2}\Big)\!\Big)
 + \Delta_\Cp \arg L(s,\chi).
\intertext{Letting $q(\chi):=Q(\chi)^{1/\nE}$ it becomes:}
&= \Delta_\Cp \arg \Big(\Big(\frac{q(\chi)}{\pi}\Big)^{s\nE/2}\Big)
 +a_\chi\Delta_\Cp \arg \Gamma\Big(\frac{s}{2}\Big)
 +b_\chi\Delta_\Cp \arg \Gamma\Big(\frac{s+1}{2}\Big)
 +\Delta_\Cp \arg L(s,\chi)\\
&= \nE T \log\Big(\frac{q(\chi)}{\pi}\Big)
 +2a_\chi\Imm\log \Gamma\Big(\frac{1}{4}+\frac{iT}{2}\Big)
 +2b_\chi\Imm\log \Gamma\Big(\frac{3}{4}+\frac{iT}{2}\Big)
 +\Delta_\Cp \arg L(s,\chi).
\end{align*}
We define the function $g(\alpha,T)$ by
\begin{equation}\label{eq:A.1}
\Imm\log\Gamma\Big(\frac{1+2\alpha}{4}+\frac{iT}{2}\Big)
=: \frac{T}{2}\log\frac{T}{2e}
 +  (2\alpha-1)\frac{\pi}{8}
 +  g(\alpha,T)
\end{equation}
for $T>0$, and by Stirling's formula we know that $g(\alpha,T) = O(1/T)$ as $T\to +\infty$. Thus, in terms
of $g(\alpha,T)$ we get
\begin{equation*}
\pi N(T,\chi)
%
 = \nE T \log\Big(\frac{q(\chi)T}{2\pi e}\Big)
 + \frac{\pi}{4}(b_\chi-a_\chi)
 + 2a_\chi g(0,T)
 + 2b_\chi g(1,T)
 + \Delta_\Cp \arg L(s,\chi).
\end{equation*}
We first show that $g(1,T) \leq g(0,T)$ for every $T\geq 0$. In fact, setting $z:= \tfrac{1}{4} +
\tfrac{iT}{2}$, by Euler's reflection formula
\[
\frac{\Gamma(\tfrac{1}{4} + \tfrac{iT}{2})}{\Gamma(\tfrac{3}{4} + \tfrac{iT}{2})}
= \frac{\Gamma(z)}{\overline{\Gamma(1-z)}}
= \frac{|\Gamma(z)|^2}{\pi\sqrt{2}}\Big(\cosh\Big(\frac{\pi T}{2}\Big)-i\sinh\Big(\frac{\pi T}{2}\Big)\Big).
\]
Since this fraction is in the fourth quadrant, this equality implies that
\[
g(0,T) - g(1,T)
= \frac{\pi}{4} + \arg\Big(\frac{\Gamma(\tfrac{1}{4} + \tfrac{iT}{2})}{\Gamma(\tfrac{3}{4} + \tfrac{iT}{2})}\Big)
= \frac{\pi}{4} - \atan\Big(\tanh\Big(\frac{\pi T}{2}\Big)\Big)
> 0.
\]
For $g(\alpha,T)$ we have the equalities:
\begin{equation}
g(\alpha,T)
= - \frac{2\alpha-1}{4}\atan\Big(\frac{2\alpha+1}{2T}\Big)
  + \frac{T}{4}\log\Big(1 + \frac{(2\alpha+1)^2}{4T^2}\Big)
  - \frac{T}{6|\frac{1}{2}+\alpha+iT|^2}
  + \frac{3\theta}{40|\frac{1}{2}+\alpha+iT|^3}             \label{eq:A.2}
\end{equation}
for some $\theta\in[-1,1]$ (see~\cite[Th.~1.4.2]{AnAsRoy}, with $m=2$),
%
%
and
\begin{multline*}
g(\alpha,T) = - \frac{2\alpha-1}{4}\atan\Big(\frac{2\alpha+1}{2T}\Big)
              + \frac{T}{4}\log\Big(1 + \frac{(2\alpha+1)^2}{4T^2}\Big)                        \\
              + \int_0^{+\infty}\Big(\frac{1}{2} - \frac{1}{t} +\frac{1}{e^t-1}\Big)\frac{e^{-(2\alpha+1)t/4}}{t} \sin\Big(\frac{tT}{2}\Big)\dd t
\end{multline*}
(see~\cite[Th.~1.6.3 (i)]{AnAsRoy}) when $2\alpha + 1> 0$. The first formula is strong enough to prove
that $g(1,T)>0$ for $T\geq 1.5$
%
%
(but an explicit computation shows that this holds also for $T\in[1,1.5]$).
%
%
The second one (with some tedious but elementary work) shows that $g(0,T)$ decreases for $T\geq 1$.
%
Therefore
\begin{equation}\label{eq:A.3}
\Big|
N(T,\chi)
 - \frac{\nE T}{\pi}\log\Big(\frac{q(\chi)T}{2\pi e}\Big)
 + \frac{a_\chi - b_\chi}{4}
\Big|
\leq \frac{2\nE}{\pi} g(0,T_0) + \frac{1}{\pi}|\Delta_\Cp \arg L(s,\chi)|
\end{equation}
for every $T\geq T_0 \geq 1$.\\
To bound $\Delta_\Cp \arg L(s,\chi)$ we split $\Cp$ in three segments $\Cp_1$, $\Cp_2$ and $\Cp_3$ where
$\Cp_2$ is the vertical one. We have
\begin{equation}\label{eq:A.4}
|\Delta_{\Cp_2} \arg L(s,\chi)|
\leq 2|\log \zeta_\E(\sigma_1)|
\leq 2\nE \log\zeta(\sigma_1).
\end{equation}
To bound $\Delta_{\Cp_1} \arg L(s,\chi)$ and $\Delta_{\Cp_3} \arg L(s,\chi)$ we apply Backlund's
argument~\cite{Backlund}, in the version given by Trudgian~\cite{TrudgianIII}.
Let
\begin{equation}\label{eq:A.5}
f(s) := \frac{1}{2}\big(L(s+iT,\chi)^N + L(s-iT,\bar\chi)^N\big)
\end{equation}
for some positive integer $N$. Suppose that there are $n$ distinct zeros of $f(\sigma){=}
\Ree(L(\sigma+iT,\chi)^N)$ for $\sigma\in[\frac{1}{2},\sigma_1]$. These zeros partition the segment into
$n+1$ intervals. On each interval $\arg(L(\sigma+iT,\chi)^N)$ can vary by at most $\pi$. Thus
\[
|\Delta_{\Cp_3} \arg L(s,\chi)| = \frac{1}{N}|\Delta_{\Cp_3} \arg L(s,\chi)^N| \leq \frac{(n+1)\pi}{N}.
\]
By symmetry the same bound applies on $\Cp_1$, thus~\eqref{eq:A.3} becomes
\begin{equation}\label{eq:A.6}
\Big|
N(T,\chi)
 - \frac{\nE T}{\pi}\log\Big(\frac{q(\chi)T}{2\pi e}\Big)
 + \frac{a_\chi-b_\chi}{4}
\Big|
\leq \frac{2\nE}{\pi} (g(0,T) + \log\zeta(\sigma_1)) + \frac{2(n+1)}{N}.
\end{equation}
In order to bound $n$ we apply Jensen's formula, see~\cite[(8)]{Jensen:nouveletimportant}
or~\cite[Th.~15.18~p.~307]{Rudin1},
\[
\log \frac{R^{m}}{|a_1a_2\cdots a_m|} = \frac{1}{2\pi}\int_0^{2\pi} \log|f(a + R e^{i\phi})|\dd\phi - \log|f(a)|
\]
where $f$ is any function which is holomorphic in the disc centred in $a$ and radius $R$, $f(a)$ is
assumed to be not zero, and $a_j$ for $j=1,\ldots,m$ is the list of all zeros of $f$ in the disc (further
assuming that there are no zeros on the boundary). We set $a=1+\eta$ with $\eta \in (0,1]$,
$R=r(\frac{1}{2}+\eta)$, $r>0$ and apply Jensen's formula to the function in~\eqref{eq:A.5}. Assuming for
the moment that $f(1+\eta)\neq 0$, \cite[Lemma~2]{TrudgianIII} (a special realization of Backlund's
trick) shows that if $\sigma_1=\frac{1}{2}+\sqrt{2}(\frac{1}{2}+\eta)$ and $1-\sigma_1 > a-R$ (which
corresponds to $r>1+\sqrt{2}$) there are $n'\geq n-2-\frac{NE}{\pi}$ real zeros in the circle and smaller
than $1/2$ which coupled with the $n$ zeros allow one to prove that
\[
\log \frac{R^m}{|a_1a_2\cdots a_m|}
\geq \log \frac{R^{n+n'}}{|a_1a_2\cdots a_{n+n'}|}
\geq(n+n')\log r
\geq 2\Big(n-1-\frac{NE}{2\pi}\Big)\log r,
\]
where $E$ is any upper bound for
\begin{equation}\label{eq:A.7}
|\Delta_+\arg L(s,\chi) + \Delta_-\arg L(s,\chi)|,
\end{equation}
where $\Delta_\pm\arg$ denotes the change of the argument between the points $\frac{1}{2}\pm \delta+iT$,
with $\delta:=\sigma_1-\frac{1}{2}$, and the point $\frac{1}{2}+iT$, proviso that
\begin{equation}\label{eq:A.8}
|\Delta_{\Cp_3}\arg L(s,\chi)^N| \geq 3\pi + NE.
\end{equation}
An argument of Heath-Brown~\cite[Subsection 3.1]{TrudgianIII} shows that the same conclusion holds also
if $\sigma_1 < a + R$ but without the assumption $1-\sigma_1 > a-R$. As a consequence, for $n$ (the number
of zeros of $f(\sigma)$ in $[\frac{1}{2},\sigma_1]$) we have the bound
\begin{equation}\label{eq:A.9}
n
\leq 1 + \frac{NE}{2\pi} + \frac{1}{4\pi\log r}\int_0^{2\pi} \log|f(a + R e^{i\phi})|\dd\phi -
\frac{1}{2\log r}\log|f(a)|,
\end{equation}
when~\eqref{eq:A.8} holds. To bound the integral, we first use the inequality $|f(s)|\leq |L(s,\chi)|^N$.
For $\phi\in[-\pi/2,\pi/2]$, we bound $L(s,\chi)$ with what we get from its representation as Dirichlet
series on the half-circle $a + R e^{i\phi}$. Thus,
\begin{align}
\frac{1}{N}\int_{-\pi/2}^{\pi/2} \log|f(a + R e^{i\phi})|\dd\phi
&\leq \frac{1}{N}\int_{-\pi/2}^{\pi/2} \log|L(a + iT + R e^{i\phi},\chi)^N|\dd\phi \notag\\
&\leq \int_{-\pi/2}^{\pi/2} \log(\zeta_\E(a + R \cos\phi))\dd\phi
 \leq \nE\int_{-\pi/2}^{\pi/2} \log(\zeta(a + R \cos\phi))\dd\phi.                 \label{eq:A.10}
\end{align}
For the remaining part of the domain, following~\cite[Subsection 4.1]{TrudgianIII}, we use Lindel\"{o}f's
convexity bound~\cite{Rademacher5} on the strip $p\leq \sigma\leq a$, where the negative parameter $p$
has to satisfy both $p\geq -1/2$ to use~\cite{Rademacher5}, and $p\leq a-R$ so that the left half-circle
is included in the strip. In fact, by~\eqref{eq:2.5},~\eqref{eq:2.6},~\eqref{eq:2.7} and~\cite[Lemmas~1,
2]{Rademacher5} we get
\begin{align*}
|L(s,\chi)|
&=    \Big(\frac{Q(\chi)}{\pi^\nE}\Big)^{\frac{1}{2}-\sigma}
      \Big|\frac{\Gamma\big(\frac{1-s}{2}\big)}{\Gamma\big(\frac{s}{2}\big)}\Big|^{a_\chi}
      \Big|\frac{\Gamma\big(\frac{2-s}{2}\big)}{\Gamma\big(\frac{1+s}{2}\big)}\Big|^{b_\chi}
      |L(1-s,\chi)|                                                                          \\
&\leq \Big(\frac{Q(\chi)}{(2\pi)^\nE}\Big)^{\frac{1}{2}-\sigma}
      |1+s|^{(\frac{1}{2}-\sigma)\nE}
      |L(1-s,\chi)|
\end{align*}
for $\sigma\in[-\frac{1}{2},\frac{1}{2}]$.
In particular, for $p\in[-\frac{1}{2},0)$
\begin{align*}
|L(p+it,\chi)|
\leq \Big(\frac{q(\chi)|1+p+it|}{2\pi}\Big)^{(\frac{1}{2}-p)\nE}
     \zeta(1-p)^{\nE}
\end{align*}
and by~\cite[Th.~2]{Rademacher5} we conclude
\[
|L(s,\chi)|
\leq \Big\{\Big(\frac{q(\chi)|1 + s|}{2\pi}\Big)^{(1/2-p)(1+\eta-\sigma)}\zeta(1-p)^{1+\eta-\sigma}\zeta(1 + \eta)^{\sigma-p}\Big\}^{\nE/(1+\eta-p)},
\]
valid for $p\leq \sigma\leq 1+\eta$ where $-\frac{1}{2}\leq p<0<\eta\leq \frac{1}{2}$. We thus have
\begin{align}
\frac{1}{N}\int_{\pi/2}^{3\pi/2} &\log|f(a + R e^{i\phi})|\dd\phi
\leq \frac{1}{N}\int_{\pi/2}^{3\pi/2} \log|L(a + iT + R e^{i\phi},\chi)^N|\dd\phi \notag\\
%
&\leq \frac{1-2p}{1+\eta-p}R\nE\log\Big(\frac{q(\chi)T}{2\pi}\Big)
  + \pi\nE \log \zeta(1+\eta)
  + \frac{2R\nE}{1+\eta-p}\log\Big(\frac{\zeta(1-p)}{\zeta(1+\eta)}\Big)          \notag\\
&\quad
  + \frac{1/2-p}{1+\eta-p}R\nE
    \int_{\pi/2}^{3\pi/2}(-\cos\phi)\log\big(w(T,\phi,\eta,R)\big) \dd\phi        \label{eq:A.11}
\end{align}
where, as in~\cite[(4.8)]{TrudgianIII} (but using $R$ instead of $r$ as the last argument of $w$)
\[
  w(T,\phi,\eta,R)^2 = 1
                     + \frac{2R\sin\phi}{T}
                     + \frac{R^2 + (2+\eta)^2 + 2R(2+\eta)\cos\phi}{T^2}.
\]
To bound this integral we use the elementary inequality $\log x \leq \frac{x^2-1}{2}$, which applied to
$w$ produces a function which can be explicitly integrated. The resulting function is decreasing in $T$,
so that it can be bounded with its value at $T_0$. With this method from~\eqref{eq:A.11} we get
\begin{align}
\frac{1}{N}&\int_{\pi/2}^{3\pi/2} \log|f(a + R e^{i\phi})|\dd\phi
\leq \frac{1-2p}{1+\eta-p}R\nE\log\Big(\frac{q(\chi)T}{2\pi}\Big)
   +  \pi \nE\log \zeta(1+\eta)                                                           \notag\\
&  +  \frac{2R}{1+\eta-p}\nE\log\Big(\frac{\zeta(1-p)}{\zeta(1+\eta)}\Big)
   +  \frac{1/2-p}{1+\eta-p}R\nE\frac{2R^2 + 2(2+\eta)^2 - \pi R(2+\eta)}{2T_0^2}         \label{eq:A.12}
\end{align}
valid for all $T\geq T_0\geq 1$, as long as $-1/2\leq p<0<\eta\leq 1/2$, $p\leq a-R$ and $\sigma_1< a+R$.
We still have to bound $-\log|f(a)|$ and for that we let $N$ diverge along a sequence such that $N\arg
L(a+iT,\chi)$ tends to $0$ modulo $2\pi$. In the limit we get
$\lim\frac{1}{N}\log|f(a)|=\log|L(a+iT,\chi)|$. We use
\begin{align}
\log|L(a+iT,\chi)|
& =  \big|\prod_{\p}\big(1-\chi(p)\Norm\p^{-a-iT}\big)^{-1}\big|
\geq \prod_{\p}\big(1+\Norm\p^{-a}\big)^{-1}                     \notag\\
& =  \prod_{p}\prod_{j=1}^{g_p}\big(1+p^{-af_j}\big)^{-1}
\geq \prod_{p}\big(1+p^{-a}\big)^{-\nE}
  =  \Big(\frac{\zeta(2a)}{\zeta(a)}\Big)^\nE.                   \label{eq:A.13}
\end{align}
In order to compute a convenient bound for $E$ in \eqref{eq:A.7}, we notice that the functional
equation~\eqref{eq:2.7} shows that $\Delta_-\arg \xi(s,\chi) = -\Delta_+\arg \xi(s,\chi)$, and that
$\Delta_\pm\arg (Q(\chi)\pi^{-\nE})^{s/2}=0$, thus~\eqref{eq:A.7} equals
\[
|\Delta_+\arg \Gamma_\chi(s) + \Delta_-\arg \Gamma_\chi(s)|.
\]
Recalling the definition of $\Gamma_\chi$ and the bound in~\eqref{eq:A.1}--\eqref{eq:A.2}, this may be
estimated by
\[
a_\chi G(0,\delta,T) + b_\chi G(1,\delta,T)
\leq \nE G(0,\delta,T)
\]
where
\begin{align*}
G(\alpha,\delta,T)
:=& \frac{1}{2}\Big(\alpha-\frac{1}{2}+\delta\Big)\atan\Big(\frac{\alpha+\frac{1}{2}+\delta}{T}\Big)
  + \frac{1}{2}\Big(\alpha-\frac{1}{2}-\delta\Big)\atan\Big(\frac{\alpha+\frac{1}{2}-\delta}{T}\Big)                 \\
& -            \Big(\alpha-\frac{1}{2}\Big)       \atan\Big(\frac{\alpha+\frac{1}{2}}{T}\Big)
  - \frac{T}{4}\log\Big(1+\frac{2\delta^2(T^2-(\frac{1}{2}+\alpha)^2)+\delta^4}{(T^2+(\frac{1}{2}+\alpha)^2)^2}\Big) \\
& + \frac{1}{4}\Big(\frac{1}{|\frac{1}{2}+\delta+\alpha+iT|} + \frac{1}{|\frac{1}{2}-\delta+\alpha+iT|} + \frac{2}{|\frac{1}{2}+\alpha+iT|}\Big)
\end{align*}
and we have used the inequalities $0 < G(1,\delta,T) \leq G(0,\delta,T)$.
%
%
Observing that $G(0,\delta,T)$ is decreasing in $T$ for $T\geq 1$, we have
\begin{equation}\label{eq:A.14}
|\Delta_+\arg L(s,\chi) + \Delta_-\arg L(s,\chi)|
\leq \nE G(0,\delta,T_0)
\end{equation}
for $T\geq T_0\geq 1$. We thus let $E := \nE G(0,\delta,T_0)$.\\
In the final inequality~\eqref{eq:A.15} the coefficient of $\log(q(\chi)T)$ is
$\frac{(1/2-p)R}{2\pi(1+\eta-p)\log r}$. It is minimal for $r=\frac{1+\eta-p}{1/2+\eta}$, hence this is
the choice we make. We then have $R=1+\eta-p$, hence $a-R=p$ and $a+R = 2+2\eta-p >
\frac{1}{2}+\sqrt{2}(\frac{1}{2}+\eta) = \sigma_1$. From~\eqref{eq:A.6}, \eqref{eq:A.9}, \eqref{eq:A.10},
\eqref{eq:A.12}, \eqref{eq:A.13} and~\eqref{eq:A.14} we have, recalling that $r=\frac{1+\eta-p}{1/2+\eta}$,
\begin{equation}\label{eq:A.15}
\Big|\frac{N(T,\chi)}{\nE}
     - \frac{T}{\pi}\log\Big(\frac{q(\chi)T}{2\pi e}\Big)
     + \frac{a_\chi-b_\chi}{4\nE}
\Big|
\leq C_1\log(q(\chi)T) + C'_2
\end{equation}
with
\begin{align}
C_1  :=& \frac{1/2-p}{\pi\log r}                                                                         \label{eq:A.16}
\intertext{and}
C'_2 :=& \frac{2}{\pi} \Big(g(0,T_0)
                            + \log\zeta\Big(\frac{1}{2}+\sqrt{2}\Big(\frac{1}{2}+\eta\Big)\Big)
                            + \frac{1}{2}G\big(0,\sqrt{2}(\frac{1}{2}+\eta),T_0\big)
                      \Big)                                                                              \notag\\
&
      + \frac{1}{2\pi\log r} \int_{-\pi/2}^{\pi/2} \log(\zeta(a + (1+\eta-p) \cos\phi))\dd\phi           \notag\\
&
      + \frac{1/2-p}{4\pi T_0^2\log r}[2(1+\eta-p)^2 + 2(2+\eta)^2 - \pi (1+\eta-p)(2+\eta)]             \notag\\
&
      - \frac{1/2-p}{\pi\log r}\log(2\pi)
      + \frac{\log \zeta(1+\eta)}{2\log r}
      + \frac{\log\big(\frac{\zeta(1-p)}{\zeta(1+\eta)}\big)}{\pi\log r}
      + \frac{1}{\log r} \log\Big(\frac{\zeta(1+\eta)}{\zeta(2(1+\eta))}\Big)                            \label{eq:A.17}
\end{align}
valid for $-1/2\leq p<0<\eta\leq 1/2$ and $T\geq T_0\geq 1$, and proviso that~\eqref{eq:A.8} holds. In
case~\eqref{eq:A.8} is false, by~\eqref{eq:A.3},~\eqref{eq:A.4} and (the opposite of)~\eqref{eq:A.8} we
still get~\eqref{eq:A.15} but with
\begin{align}
C_1  &:= 0,                                                                                              \label{eq:A.18}\\
C'_2 &:= \frac{2}{\pi} \Big(g(0,T_0)
                            + \log\zeta\Big(\frac{1}{2}+\sqrt{2}\Big(\frac{1}{2}+\eta\Big)\Big)
                            + G\big(0,\sqrt{2}(\frac{1}{2}+\eta),T_0\big)
                      \Big).                                                                             \label{eq:A.19}
\end{align}
To obtain the values in Table~\ref{tab:A1}, we observe that by~\eqref{eq:A.16} we have
\[
\eta=\frac{1/2-p}{\exp((1/2-p)/(\pi C_1))-1}-\frac{1}{2}
\]
for every given $C_1$ and $p\in[-\frac{1}{2},0)$.

Coming to the case where $\chi$ is trivial, we follow the proof of \cite[Theorem~2]{TrudgianIII} with the
modifications we have made above, and we observe that $\Delta_{\Cp}s(s-1)=2\pi$, which accounts for the
$-2\delta_\chi$ in the main term of $N(T,\chi)=N_\L(T)$.

For the remaining terms, we observe that $g(T):=\Imm\log\Gamma(1/2+iT)-T\log(T/e)$ and $g(0,T)$ both
decrease to $0$ as $T\to\infty$,
and that $g(T)\leq g(0,T)$, hence we can use $D_1:=C_1$ and $D'_2 := C'_2$.\\
Moreover using that $\log x \leq (x^2-1)/2$ to bound the integrals in the expression of $D_3$
of~\cite[(5.12)]{TrudgianIII}, we can use
%
%
\[
D'_3 := \frac{1}{\pi\log r}\log\Big(\frac{1-p}{1+p}\Big)
       + \frac{1}{\pi}F\big(\sqrt{2}\big(\tfrac{1}{2}+\eta\big),T_0\big)
       + \frac{\pi r^2(\frac{1}{2}+\eta)^2
               - 4r(\frac{1}{2}+\eta)
               + \pi\eta^2
               + 2\pi\eta
               + 2\pi}{2\pi T_0^2\log r}
\]
where $F(\delta,T):=2\atan\frac{1}{2T} - \atan\frac{1/2+\delta}{T} - \atan\frac{1/2-\delta}{T}$.

\enlargethispage{3\baselineskip}
We use the formula given above for $\eta$ in terms of $C_1=D_1$ and $p$, we compute the values of
$D'_2=C'_2$ for a suitable choice of $p$ as given by~\eqref{eq:A.17} and we test that it is greater than
the value produced by~\eqref{eq:A.19}; an upper bound for $D'_2$, a rounding of the computed value of
$\eta$ and the chosen value of $p$ are indicated in the table below (the sequences of values of $D_1$ are
the same in the three subtables and are those indicated in~\cite[Table~2]{TrudgianIII}, plus the two
extremal values $0.230$ and $0.460$).
\clearpage

\enlargethispage{3\baselineskip}
{\small
\[
\begin{array}[b]{r|rrll|rrll}
  \toprule
 \mltc{1}{c|}{}   &\mltc{4}{c|}{T_0=1}                                                     &\mltc{4}{c}{T_0=2\pi}                                                  \\
 \mltc{1}{c|}{D_1}&\mltc{1}{c}{D'_2} &\mltc{1}{c}{D'_3} &\mltc{1}{c}{\eta} &\mltc{1}{c|}{p}&\mltc{1}{c}{D'_2} &\mltc{1}{c}{D'_3} &\mltc{1}{c}{\eta} &\mltc{1}{c}{p}\\
  \midrule
0.230             &           16.577 &           1.330  &         0.00090  &     -0.00070  &           16.032 &           0.033  &         0.00090  &     -0.00070 \\
0.247             &            8.180 &           1.435  &         0.03058  &     -0.05681  &            7.614 &           0.083  &         0.03111  &     -0.05542 \\
0.265             &            6.416 &           1.515  &         0.05175  &     -0.14367  &            5.834 &           0.150  &         0.05390  &     -0.13792 \\
0.282             &            5.409 &           1.598  &         0.06920  &     -0.23355  &            4.812 &           0.213  &         0.07236  &     -0.22490 \\
0.299             &            4.696 &           1.699  &         0.08646  &     -0.32500  &            4.083 &           0.275  &         0.09004  &     -0.31500 \\
0.316             &            4.158 &           1.814  &         0.10280  &     -0.42000  &            3.526 &           0.335  &         0.10982  &     -0.40000 \\
0.333             &            3.735 &           1.961  &         0.12462  &     -0.50000  &            3.082 &           0.400  &         0.12808  &     -0.49000 \\
0.350             &            3.425 &           2.185  &         0.17432  &     -0.50000  &            2.731 &           0.429  &         0.17432  &     -0.50000 \\
0.367             &            3.206 &           2.426  &         0.22435  &     -0.50000  &            2.467 &           0.453  &         0.22435  &     -0.50000 \\
0.384             &            3.043 &           2.687  &         0.27467  &     -0.50000  &            2.257 &           0.478  &         0.27467  &     -0.50000 \\
0.401             &            2.918 &           2.966  &         0.32520  &     -0.50000  &            2.083 &           0.503  &         0.32520  &     -0.50000 \\
0.460             &            2.666 &           4.082  &         0.50000  &     -0.50000  &            1.645 &           0.593  &         0.50000  &     -0.50000 \\
  \bottomrule
\end{array}
\]
\[
\begin{array}[b]{r|rrll}
  \toprule
 \mltc{1}{c|}{}   &\mltc{4}{c}{T_0=10}                                                     \\
 \mltc{1}{c|}{D_1}&\mltc{1}{c}{D'_2} &\mltc{1}{c}{D'_3} &\mltc{1}{c}{\eta} &\mltc{1}{c}{p} \\
  \midrule
0.230             &           16.004 &           0.014  &         0.00091  &     -0.00067  \\
0.247             &            7.585 &           0.062  &         0.03164  &     -0.05404  \\
0.265             &            5.805 &           0.129  &         0.05390  &     -0.13792  \\
0.282             &            4.783 &           0.192  &         0.07236  &     -0.22490  \\
0.299             &            4.053 &           0.254  &         0.09004  &     -0.31500  \\
0.316             &            3.495 &           0.313  &         0.10982  &     -0.40000  \\
0.333             &            3.050 &           0.371  &         0.13156  &     -0.48000  \\
0.350             &            2.698 &           0.402  &         0.17432  &     -0.50000  \\
0.367             &            2.432 &           0.423  &         0.22435  &     -0.50000  \\
0.384             &            2.221 &           0.444  &         0.27467  &     -0.50000  \\
0.401             &            2.044 &           0.465  &         0.32520  &     -0.50000  \\
0.460             &            1.598 &           0.540  &         0.50000  &     -0.50000  \\
  \bottomrule
\end{array}
\qedhere
\]
}
\end{proof}

\begin{fixedtab}
\caption{Constants for Lemma~\ref{lem:3.9}.}\label{tab:A2}
\smallskip
{\scriptsize
\tabcolsep=2pt
\begin{tabular}{|r|r||r|r|}
  \toprule
  $j$  & \mltc{1}{c||}{$a_j\cdot 10^7$}          &  $j$  &  \mltc{1}{c|}{$a_j\cdot 10^7$}             \\
  \midrule
  $ 1$ & $                              67441107$&  $26$ & $   4711532246020032770961059850536842961$ \\
  $ 2$ & $                          129064216397$&  $27$ & $  -9979971210677326363399566081587309621$ \\
  $ 3$ & $                       -33671827706277$&  $28$ & $  19147233119732562826091118305794764779$ \\
  $ 4$ & $                      4159437592468632$&  $29$ & $ -33274047709559371113992775342599269485$ \\
  $ 5$ & $                   -315432926321374242$&  $30$ & $  52358220195286687433763798635287630555$ \\
  $ 6$ & $                  16370077474919646336$&  $31$ & $ -74548381119823637972378393085833994786$ \\
  $ 7$ & $                -620228745134606597597$&  $32$ & $  95937426238030011573589993986867291432$ \\
  $ 8$ & $               17934517713943067903261$&  $33$ & $-111421834266414109909340554112526772452$ \\
  $ 9$ & $             -408973952667945326004549$&  $34$ & $ 116550516507798376160362309501875288819$ \\
  $10$ & $             7542955862267902755091933$&  $35$ & $-109525478172827789046963052436691334874$ \\
  $11$ & $          -114797714164799489558618807$&  $36$ & $  92171825266689255311105390157515626975$ \\
  $12$ & $          1465278757842284478556905563$&  $37$ & $ -69194087310394938615774929447065136471$ \\
  $13$ & $        -15896327170655789866055422304$&  $38$ & $  46115594804031958286535245249055216023$ \\
  $14$ & $        148210358380111290581087608810$&  $39$ & $ -27125577798003271571724298417346235682$ \\
  $15$ & $      -1198675077750183343628567667972$&  $40$ & $  13979915122173412958783020578998059040$ \\
  $16$ & $       8475563352018452380288345356252$&  $41$ & $  -6255803435136676900876694551147848415$ \\
  $17$ & $     -52742543205461653283881602845090$&  $42$ & $   2402825607446165955037188420531530836$ \\
  $18$ & $     290485125582627204720553754700530$&  $43$ & $   -780487429206171024872362699598861667$ \\
  $19$ & $   -1422762853575378758435389963062636$&  $44$ & $    210196127819906522271561747433766713$ \\
  $20$ & $    6222222002869884071289885659404750$&  $45$ & $    -45668115875651680795706979313599659$ \\
  $21$ & $  -24380706266315957556815280817594915$&  $46$ & $      7690167072902618888205802917980935$ \\
  $22$ & $   85837343646704274150965262557412097$&  $47$ & $      -941636162712117945732981144066824$ \\
  $23$ & $ -272183051338763525712916125735803989$&  $48$ & $        74577580991057238830195411510057$ \\
  $24$ & $  778809192744980652056346184699871878$&  $49$ & $        -2867250294949111291564065810976$ \\
  $25$ & $-2013896299428527154913597515037117583$&       &                                            \\
  \bottomrule
\end{tabular}
}
\end{fixedtab}

\bibliographystyle{amsplain}

\end{document}